\documentclass[a4paper, 12pt]{amsart}
\usepackage{amssymb}
\usepackage{amsfonts}
\usepackage{amsthm}
\usepackage{amsmath}
\usepackage{color}
\usepackage{amscd}
\usepackage{fullpage}
\usepackage{hyperref}

\usepackage{graphicx}
\usepackage[all]{xy}
\usepackage{array}
\usepackage{youngtab}

\newcommand{\und}{\underline}

\newcommand{\BN}{{\mathbb {N}}}

\newcommand{\BQ}{{\mathbb {Q}}}

\newcommand{\CA}{{\mathcal {A}}}

\newcommand{\CC}{{\mathcal {C}}}

\newcommand{\CF}{{\mathcal {F}}}

\newcommand{\CI}{{\mathcal {I}}}

\newcommand{\CM}{{\mathcal {M}}}

\newcommand{\CO}{{\mathcal {O}}}
\newcommand{\CP}{{\mathcal {P}}}

\newcommand{\CR}{{\mathcal {R}}}
\newcommand{\CS}{{\mathcal {S}}}

\newcommand{\RA}{{\mathrm {A}}}

\newcommand{\RI}{{\mathrm {I}}}
\newcommand{\RJ}{{\mathrm {J}}}

\newcommand{\ind}{{\mathrm{ind}}}

\newcommand{\Hom}{{\mathrm{Hom}}}

\newcommand{\Lie}{{\mathrm{Lie}}}

\newcommand{\rk}{{ \mathrm{k}  }}

\newcommand{\rr}{{\mathrm{R}}}

\newcommand{\con}{\textit{C}}

\newcommand{\Ext}{\operatorname{Ext}}
\renewcommand{\span}{\operatorname{span}}

\newcommand{\od}{\operatorname{d}}

\newcommand{\oJ}{\operatorname{J}}
\newcommand{\oH}{\operatorname{H}}

\newcommand{\oS}{\mathcal{S}}

\newcommand{\oT}{\operatorname{T}}

\newcommand{\oZ}{\operatorname{Z}}
\newcommand{\oD}{\textit{D}}

\newcommand{\oM}{\operatorname{M}}

\renewcommand{\rk}{\mathrm F}

\newcommand{\Z}{\mathbb{Z}}
\newcommand{\C}{\mathbb{C}}
\newcommand{\R}{\mathbb R}

\newcommand{\aut}{\mathcal A}
\newcommand{\abs}[1]{\lvert#1\rvert}

\newcommand{\cf}{\emph{cf.}~}

\newcommand{\be}{\begin{equation}}
\newcommand{\ee}{\end{equation}}
\newcommand{\bee}{\begin {equation*}}
\newcommand{\eee}{\end {equation*}}

\theoremstyle{Theorem}

\theoremstyle{Theorem}
\newtheorem{lem}{Lemma}[section]
\newtheorem{corl}[lem]{Corollary}
\newtheorem{thml}[lem]{Theorem}

\newtheorem{prpl}[lem]{Proposition}

\theoremstyle{Theorem}
\newtheorem{prp}{Proposition}[section]

\theoremstyle{Plain}
\newtheorem*{remark}{Remarks}

\theoremstyle{Definition}
\newtheorem{dfn}{Definition}[section]

\newtheorem{dfnl}[lem]{Definition}
\newtheorem{cord}[dfn]{Corollary}
\newtheorem{prpd}[dfn]{Proposition}
\newtheorem{thmd}[dfn]{Theorem}
\newtheorem{lemd}[dfn]{Lemma}

\newtheorem{theorem}{Theorem}[section]

\newtheorem{proposition}[theorem]{Proposition}

\theoremstyle{definition}

\begin{document}

\title{Generalized  semi-invariant distributions on $p$-adic spaces}

\author{Jiuzu Hong}
\address{Department of Mathematics, University of North Carolina at Chapel Hill, CB 3250 Phillips Hall Chapel Hill, NC 27599-3250, U.S.A.} \email{jiuzu@email.unc.edu}
\author{Binyong Sun}
\address{Hua Loo-Keng Key Laboratory of Mathematics, Institute of Mathematics,
AMSS, Chinese Academy of Sciences,Beijing, 100190, P.R. China.}
\email{sun@math.ac.cn}


\date{}

\maketitle

\begin{abstract}
In this  paper we investigate some methods on calculating  the spaces
of generalized semi-invariant  distributions on $p$-adic spaces.
Using homological methods, we  give a criterion of automatic extension of (generalized) semi-invariant distributions. Based on the meromorphic continuations of
Igusa zeta integrals, we give another criteria with purely algebraic
geometric conditions,  on the extension of generalized semi-invariant  distributions.
\end{abstract}

\tableofcontents

\section{Introduction}

Following Bernstein-Zelevinsky \cite{BZ}, we define an $\ell$-space
to be a topological space which is  Hausdorff, locally compact,
totally disconnected and secondly countable.  An $\ell$-group is a
topological group whose underlying topological space is an
 $\ell$-space. Let $G$ be an $\ell$-group acting continuously on an $\ell$-space $X$. We may ask a general question about how to  describe all semi-invariant distributions on $X$ with respect to the action of $G$, that is, to determine the space
 \be\label{rchi}
 \oD(X)^\chi:=  \Hom_G(\oS(X),\chi)
 \ee
 for a fixed character $\chi: G\to \C^\times$ (all characters of $\ell$-groups are assumed to be locally constant in this paper). Here $\oS(X)$ denotes the space of Bruhat-Schwartz functions on $X$, namely, the space of compactly supported, locally constant complex valued functions on $X$.
Here and as usual, when no confusion is possible, we do not
distinguish a representation with its underlying (complex) vector
space. In particular, we do not distinguish a character with the
representation  attached to it on the one-dimensional vector space
$\C$.
  We call an element of \eqref{rchi} a $\chi$-invariant distribution on $X$. Many problems on number theory and representation theory of $p$-adic groups end up to the problems on semi-invariant distributions of this kind.  There are quite a lot of techniques on the vanishing of invariant distributions.  It seems to us that the constructions of semi-invariant distributions are still not fully developed.

We suggest in this paper that, to  describe all semi-invariant
distributions on the $\ell$-spaces, it would be  more achievable to first
consider some more general distributions. They are the generalized
semi-invariant distributions as in the following definition.

\begin{dfn}
Let $V$ be a (non-necessary smooth) representation of $G$. A vector
$v\in V$ is called a generalized invariant vector if there is a
$k\in \BN$ such that
\[
  (g_0-1)(g_1-1)\cdots (g_k-1).v=0\quad \textrm{for all }g_0, g_1, \cdots, g_k\in G.
\]
A generalized $\chi$-invariant distribution on $X$ is defined to be
a  generalized invariant vector in the representation
$\Hom_\C(\oS(X), \chi)$ of $G$.
\end{dfn}

Here and as usual, the group $G$ acts on  $\Hom_\C(\oS(X),
\chi)$ as in  the equation \eqref{actint} of Section \ref{secgh}.
The set of non-negative integers is denoted by $\BN$.

When $X$ is a $G$-homogeneous space (to be more precise, this means
that the action of $G$ on $X$ is transitive, and for every $x\in X$,
the orbit map $G\to X, \,g\mapsto g.x$ is open), the space
\eqref{rchi} is at most one-dimensional. We introduce the following
definition.

\begin{dfn}
When $X$ is a homogeneous space of $G$, we say that $X$ is
$\chi$-admissible if the space \eqref{rchi} is non-zero.
\end{dfn}

We are mainly concerned with $\ell$-spaces and $\ell$-groups of
algebraic geometric origin. Throughout the paper, we fix a
non-archimedean local field $\rk$ of characteristic zero.

\begin{dfn}\label{weakad}
Assume that $G=\mathsf G(\rk)$ for some linear algebraic group $\mathsf G$ defined over $\rk$. Let $\mathsf X$ be an algebraic variety over $\rk$, with an algebraic action of $\mathsf G$. We say that a $G$-orbit $O\subset \mathsf X(\rk)$ is weakly $\chi$-admissible if the homogeneous space $G/\mathsf G_x^\circ(\rk)$ is $\chi$-admissible, where $x\in O$,   and ${\mathsf G}_x^\circ$ denotes the identity connected component of the stabilizer $\mathsf G_x$ of $x$ in $\mathsf G$.
\end{dfn}

The above definition is certainly independent of the choice of $x\in O$. As usual, by an algebraic variety over $\rk$, we mean a scheme over $\rk$
which is separated, reduced, and of finite type. A linear algebraic
group over $\rk$ is a group scheme over $\rk$ which is an affine
variety as a scheme.

The first
main result we obtain in this paper is the following automatic
extension theorem for semi-invariant distributions and
generalized semi-invariant distributions.

\begin{thmd}\label{main1}
Let $\mathsf G$ be a linear algebraic group defined over $\rk$,
acting algebraically on an algebraic variety $\mathsf X$ over
$\rk$. Let $\chi$ be a character of $\mathsf G(\rk)$, and let
$\mathsf U$ be a $\mathsf G$-stable open subvariety of
$\mathsf X$. Assume that every $\mathsf G(\rk)$-orbit in
$(\mathsf X\setminus \mathsf U)(\rk)$ is not weakly $\chi$-admissible.
Then every $\chi$-invariant distribution on $\mathsf U(\rk)$
uniquely extends to a $\chi$-invariant distribution on $\mathsf
X(\rk)$, and every generalized $\chi$-invariant distribution on
$\mathsf U(\rk)$ uniquely extends to a generalized
$\chi$-invariant distribution on $\mathsf X(\rk)$.
\end{thmd}

In Theorem \ref{main1}, if we replace ``weakly $\chi$-admissible" by ``$\chi$-admissible", then the uniqueness assertion of the theorem  remains true, by the localization principle of Bernstein-Zelevinsky \cite[Theorem 6.9]{BZ}.  In particular it implies that if every $\mathsf G(\rk)$-orbit in $\mathsf X(\rk)$ is not $\chi$-admissible, then there is no nonzero   generalized $\chi$-invariant distribution on $\mathsf X(\rk)$.
But the extendability may fail in general, as shown in the following example. Let $G=\{\pm 1\}\ltimes \rk^\times$, which acts on $X:=\{(x,y)\in \rk^2\mid xy=0\}$ by
\[
   (1,a).(x,y):=(ax, a^{-1}y)\quad \textrm{and}\quad (-1,a).(x,y):=(a^{-1}y, ax),
\]
for all $a\in \rk^\times$ and $(x,y)\in X$. Let $\chi$ be the non-trivial quadratic character of $G$ which is trivial of $\rk^\times$. Then the orbit $\{(0,0)\}$ is weakly $\chi$-admissible, but not $\chi$-admissible. It is well known that a non-zero $\chi$-invariant distribution on $X\setminus\{(0,0)\}$ does not extends to a  $\chi$-invariant distribution on $X$.

The idea of generalized semi-invariant distributions can even be
dated back to the famous Tate's thesis. It has rooted in the
dimension one property of the space of semi-invariant
distributions on $\rk$ with respect to the multiplicative action of
$\rk^\times$. Let $\chi$ denote a character of $\rk^\times$ for the
moment. As a simple application of   Theorem \ref{main1}, we know
that \be\label{dimtate}
  \dim \Hom_{\rk^\times}(\oS(\rk), \chi)=1
\ee when $\chi$ is non-trivial. However, Theorem \ref{main1} is no
longer applicable when $\chi$ is trivial. Instead, when $\chi$ is
trivial, we consider the meromorphic continuation of the zeta
integral
$$\int_\rk \phi(x) |x|^s dx,\qquad \phi\in \oS(\rk).$$
This zeta integral has simple pole at $s=-1$. Taking  all
coefficients of the Laurent expansion of the zeta integral at
$s=-1$, we actually get all generalized $\chi$-invariant
distributions on $\rk$. By considering the natural action of
$\rk^\times$ on this space of all generalized $\chi$-invariant
distributions, one concludes that \eqref{dimtate} also holds when
$\chi$ is trivial.

A key observation of the above argument is that  generalized
invariant distributions on $\rk^\times$ extends to  generalized
invariant distributions on $\rk$. The second main result of this
paper is the following generalization of this observation.
\begin{thmd}\label{main2}
Let $\mathsf G$ be a linear algebraic group over $\rk$. Let
$\mathsf X$ be  an algebraic variety over $\rk$  so that
$\mathsf G$ acts algebraically  on it with an open orbit
$\mathsf U\subset \mathsf X$. Assume that there is a  semi-invariant
regular function $f$ on  $\mathsf X$,   with the
following properties:
\begin{itemize}
  \item $f$ does not vanish on $\mathsf U$, and $\mathsf X_f\setminus \mathsf U$ has codimension $\geq 2$ in  $\mathsf X_f$, where $\mathsf X_f$ denotes the complement  in $\mathsf X$ of the zero locus of $f$;
  \item the variety $\mathsf X_f$ has  Gorenstein rational singularities.
\end{itemize}
 Let $\chi$ be a character of $\mathsf G(\rk)$ which is trivial on $\mathsf N(\rk)$, where $\mathsf N$ denotes the unipotent radical of $\mathsf G$.
Then every generalized $\chi$-invariant distribution on $\mathsf
U(\rk)$ extends to a  generalized $\chi$-invariant distribution on
$\mathsf X(\rk)$.
\end{thmd}

Here a  regular function $f$ on $\mathsf X$ being semi-invariant means that, there exists an algebraic character $\nu$ of $\mathsf G$ over $\rk$ such that
\begin{equation}
\label{Semi-invariant}
  f(g.x)=\nu(g) f(x), \quad \textrm{for all  $g\in \mathsf G(\bar \rk)$ and $x\in \mathsf X(\bar \rk)$},
\end{equation}
where $\bar \rk$ denotes an algebraic closure of $\rk$.
\begin{remark}
(a) Let $\mathsf G$ be a linear algebraic group over $\rk$, acting
algebraically on an algebraic variety $\mathsf Y$ over $\rk$. We
say that $\mathsf Y$ is $\mathsf G$-homogeneous, or $\mathsf
Y$ is a $\mathsf G$-homogeneous space, if the action of $\mathsf
G(\bar \rk)$ on $\mathsf Y({\bar \rk})$ is transitive.  In general, a subvariety
$\mathsf Z$ of $\mathsf Y$ is called a $\mathsf G$-orbit if it
is $\mathsf G$-stable and $\mathsf G$-homogeneous.

(b) We say that a subvariety $\mathsf Z$ of an algebraic variety
$\mathsf Y$ has codimension $\geq r$ ($r\in \BN$) if
\[
r+\dim_x\mathsf Z\leq \dim_x \mathsf Y\quad \textrm{ for all
$x\in \mathsf Z$.} \]

(c) The notion of Gorenstein rational singularity is reviewed in
Section \ref{secalg}.

(d)  A variant of Theorem \ref{main2} is stated in Theorem \ref{variant2}, where $\mathsf X_f$ is only required to have rational singularities, but we additionally assume that there exists a nonzero semi-invariant algebraic volume form on $U$.
\end{remark}

In order to prove Theorem \ref{main2},  as in the case of Tate's thesis we need to employ the theory of  zeta integrals.
For each generalized $\chi$-invariant distribution $\mu$ on $\mathsf U(\rk)$, it turns out that $\mu$ is a definable measure (Definition \ref{def-measure}) and it is locally finite on $\mathsf X_f(\rk)$ (Theorem \ref{convg0}). We  attach a zeta integral
 $$\oZ_{\mu,f}(\phi,s):=\int_{\mathsf X_f(\rk)} \phi(x)|f(x)|^s \,\od \!\mu(x) , $$
for every $\phi\in \oS(\mathsf X(\rk))$.  The meromorphic continuation of $\oZ_{\mu,f}$ is a consequence of a general fact of Igusa zeta integrals on semi-algebraic spaces, which is proved in Theorem \ref{igusa}.

The structure of the paper is as follows. In Section \ref{secginv}, we introduce the basics of generalized homomorphisms and generalized extensions, and we prove a vanishing theorem  of generalized extensions (Theorem \ref{Vanishing_Theorem}).
In Section \ref{loc_princ}, we prove a localization principle for extensions in the settings of equivariant $\ell$-sheaves (Theorem \ref{Fiber_Vanishing}).
Section \ref{secaut} is devoted to a proof of our first main theorem. We first establish the generalized version of Frobenius reciprocity and Shapiro Lemma. Then by results in Section \ref{secginv} and Section \ref{loc_princ}, we prove a higher version of automatic extension theorem (Theorem \ref{autoext}), which contains Theorem \ref{main1} as a special case.

In Section \ref{sec_Igusa}, we introduce $p$-adic semi-algebraic spaces and the measure theory on them. We prove the meromorphic continuation of Igusa zeta integral on general semi-algebraic spaces (Theorem \ref{igusa}) after the works of Denef, Cluckers et al.
In Section \ref{sec_main2}, we prove the second main theorem as follows. We  first prove that any generalized semi-invariant distribution on algebraic homogeneous spaces is a definable measure (Theorem \ref{defm}, Proposition \ref{defm3}), in the sense of Definition \ref{def-measure}. Then we prove that it is locally finite  (Theorem \ref{convg0}) if the boundary has Gorenstein rational singularities. In the end Theorem \ref{main2}  follows from Theorem \ref{igusa}.

As an illustration,  we determine all generalized semi-invariant distributions on matrix spaces in Section \ref{example}. In this example  we employ intensively the automatic extension theorem and meromorphic continuations of distributions.   \\

\noindent {\bf Acknowledgements}: J.\,Hong would like to thank Rami
Aizenbud, Joseph Bernstein and Yiannis Sakellaridis    for helpful
discussions. He also would like to thank AMSS, Chinese academy of
science for the hospitality during his two visits in July-August and
December of 2013,  where part of the work was done. B.\,Sun was
supported by the NSFC Grants 11525105, 11321101, and 11531008. Finally, both authors would like to thank the anonymous referee for many valuable comments which have led to an improvement of this paper.

\section{Generalized homomorphisms and generalized extensions}\label{secginv}
\subsection{The space of generalized invariant vectors}
Let $G$ be an $\ell$-group as in the Introduction. By a
representation of $G$, we mean a complex vector space together with
a linear action of $G$ on it. A vector in a representation of $G$ is
said to be smooth if it is fixed by an open subgroup of $G$.  A
representation of  $G$ is said to be smooth if all its vectors are
smooth.

Let $V$ be a representations of $G$. Define a sequence
\[
  V^{G,0}\subset V^{G,1}\subset V^{G,2}\subset\cdots
\]
of subrepresentations of $V$ by \be\label{geninv}
  V^{G,k}:=\{v\in V\mid (g_0-1)(g_1-1)\cdots (g_k-1).v=0\,\textrm{ for all }g_0, g_1,\cdots, g_k\in G\}.
\ee Put
\[
  V^{G,\infty}:=\bigcup_{k\in \BN}  V^{G,k}.
\]
A  vector of $V^{G,\infty}$ is called a generalized $G$-invariant
vector in $V$.
\begin{dfnl}
A  representation of $G$ is said to be locally unipotent if it is
smooth and all its vectors are generalized $G$-invariant.
\end{dfnl}

At least when $V$ is a smooth representation, it is elementary to
see that every compact subgroup of $G$ acts trivially on
$V^{G,\infty}$. Define $G^\circ$ to be the subgroup of $G$ generated
by all compact subgroups of $G$, which is an open normal subgroup of
$G$ (similar notation will be used without further explanation for other $\ell$-groups). Put
\[
  \Lambda_G:=G/G^\circ.
\]
Then when $V$ is smooth, $V^{G,\infty}$ descends to a locally
unipotent representation of $\Lambda_G$.

Recall from the Introduction that $\rk$ is a non-archimedean local
field of characteristic zero.

\begin{prp}(see \cite[Chapter II, Proposition 22]{Be})
Assume that $G=\mathsf G(\rk)$ for some connected linear algebraic
group $\mathsf G$ defined over $\rk$. Then $\Lambda_G$ is a free abelian group
whose rank equals the dimensional of the maximal central split torus
of a Levi component of $\mathsf G$.

\end{prp}

\subsection{Generalized homomorphisms}
\label{secgh}

Let $V_1$, $V_2$ be two smooth representations of $G$. Then
$\Hom_\C(V_1,V_2)$ is naturally a representation of $G$:
\be\label{actint}
  g.\phi(v):= g.(\phi(g^{-1}.v)),\quad \phi\in \Hom_\C(V_1,V_2),\, v\in V_1.
\ee For each $k=0, 1,2,\cdots, \infty$, put
\[
   \Hom_{G,k}(V_1,V_2):=(\Hom_{\C}(V_1,V_2))^{G,k}.
\]
We call a vector in $\Hom_{G,\infty}(V_1,V_2)$ a generalized
homomorphism from $V_1$ to $V_2$.

\begin{lem}\label{homg}
For each open compact subgroup $K$ of $G$, one has that
\[
  \Hom_{G,\infty}(V_1,V_2)\subset \Hom_K(V_1,V_2).
\]
In particular, every generalized homomorphism from $V_1$ to $V_2$ is
a smooth vector of $\Hom_{\C}(V_1,V_2)$.
\end{lem}
  \begin{proof}
Write
\[
 V_1=\bigoplus_{i\in I} V_{1,i}
\]
as a direct sum of finite dimensional representations of $K$. Then
one has that \begin{eqnarray*}
                && \Hom_{G,\infty}(V_1,V_2) \\
                &\subset& \Hom_{K,\infty}(V_1,V_2) \\
                &\subset & \prod_{i\in I}   \Hom_{K,\infty}(V_{1,i},V_2) \\
                &=& \prod_{i\in I}   \Hom_{K}(V_{1,i},V_2) \qquad(\textrm{since $\Hom_{\C}(V_{1,i},V_2)$ is a smooth representation of $K$})\\
                &=& \Hom_{K}(V_1,V_2). \\
             \end{eqnarray*}

  \end{proof}

By Lemma \ref{homg}, we know that $\Hom_{G,k}(V_1,V_2)$ is a locally
unipotent representation of $\Lambda_G$ ($k=0,1,2,\cdots, \infty$).
The following lemma is obvious.
\begin{lem}\label{homg2}
One has that
\[
  \Hom_{G,\infty}(V_1,V_2)=0\quad\textrm{if and only if}\quad \Hom_{G}(V_1,V_2)=0.
\]
\end{lem}
The following lemma is routine to check. We omit the details.

\begin{lem}
(a) Let $V_1, V_2, V_3$ be smooth representations of $G$. Let $k_1,
k_2\in\{0,1,2,\cdots, \infty\}$. Then
\[
  \phi_2\circ\phi_1\in \Hom_{G,k_1+k_2}(V_1,V_3)
\]
for all $\phi_1\in \Hom_{G,k_1}(V_1,V_2)$ and $\phi_2\in
\Hom_{G,k_2}(V_2,V_3)$.

(b) Let $V_1, V_2, V_1', V_2'$ be smooth representations of $G$. Let
$k_1, k_2\in\{0,1,2,\cdots, \infty\}$. Then
\[
  \phi_1\otimes \phi_2\in \Hom_{G,k_1+k_2}(V_1\otimes V_2,V_1'\otimes V_2')
\]
for all $\phi_1\in \Hom_{G,k_1}(V_1,V_1')$ and $\phi_2\in
\Hom_{G,k_2}(V_2,V_2')$.
\end{lem}

\subsection{Generalized homomorphisms and homomorphisms}
Denote by $\C[\Lambda_G]$ the group algebra of $\Lambda_G$. Denote
by $\RI_{G}$ the augmentation ideal of $\C[\Lambda_G]$,
namely,
\[
  \RI_{G}:=\left\{\sum_{g\in \Lambda_G} a_g\, g\in \C[\Lambda_G]
  \mid \sum_{g\in \Lambda_G} a_g=0\right\}.
\]
For each $k\in \BN$, put
\[
  \RJ_{G,k}:=\C[\Lambda_G]/(\RI_{G})^{k+1}.
\]
We view it as a locally unipotent representation of $G$ through left
translations. The following lemma is routine to check.
\begin{lem}\label{ginv}
Let $k\in \BN$. For each smooth representation $V$ of $G$, the map
\[
  \Hom_G(\RJ_{G,k},V)\rightarrow V^{G,k},\quad \phi\mapsto \phi(1)
\]
is a well-defined isomorphism of locally unipotent representations
of $G$. Here $\Hom_G(\RJ_{G,k},V)$ is viewed as a smooth
representation of $G$ by
\[
  (g.\phi)(x):=\phi(x \bar g), \quad g\in G,\, \phi\in
  \Hom_G(\RJ_{G,k},V),\, x\in\RJ_{G,k},
\]
where $\bar g$ denotes the image of $g$ under the natural map
$G\rightarrow \RJ_{G,k}$.

\end{lem}

More generally, we have the following lemma.
\begin{lem}\label{ginv2}
Let $k\in \BN$. For all smooth representations $V_1$ and $V_2$ of
$G$, the map
\[
  \Hom_G(\RJ_{G,k}\otimes V_1,V_2)\rightarrow \Hom_{G,k}(V_1,V_2),\quad \phi\mapsto \phi|_{V_1}
\]
is a well-defined isomorphism of locally unipotent representations
of $G$. Here $V_1$ is identified with the subspace $1\otimes V_1$ of
$\RJ_{G,k}\otimes V_1$, and $\Hom_G(\RJ_{G,k}\otimes V_1,V_2)$ is
viewed as a smooth representation of $G$ by
\[
  (g.\phi)(x\otimes v):=\phi(x\bar g\otimes v), \quad g\in G,\, \phi\in
  \Hom_G(\RJ_{G,k},V),\, x\in\RJ_{G,k},\,v\in V_1,
\]
where $\bar g$ denotes the image of $g$ under the natural map
$G\rightarrow \RJ_{G,k}$.

\end{lem}

\begin{proof}
We have the $G$-equivariant identifications
\begin{eqnarray*}
  \Hom_{G,k}(V_1,V_2) &=& \Hom_\C(V_1,V_2)^{G,k}\\
   &=& \Hom_G(\RJ_{G,k}, \Hom_\C(V_1,V_2)) \\
  &=& \Hom_G(\RJ_{G,k}\otimes V_1,V_2).
\end{eqnarray*}
Therefore the lemma follows.
\end{proof}

Lemma \ref{ginv} implies that
\[
 V^{G,\infty}=\varinjlim_k  \Hom_G(\RJ_{G,k},V),
\]
for all smooth representations $V$ of $G$. Likewise, Lemma
\ref{ginv2} implies that
\[
 \Hom_{G,\infty}(V_1,V_2)=\varinjlim_k \Hom_G(\RJ_{G,k}\otimes
 V_1,V_2),
\]
for all smooth representations $V_1$ and $V_2$ of $G$.

\subsection{Schwartz inductions}\label{secshw} We briefly recall the Schwartz inductions in this subsection.
Let $H$ be a closed subgroup of  $G$. Let $V_0$ be a smooth
representation of $H$. Define the un-normalized Schwartz induction
$\ind_H^G V_0$ to be the space of all $V_0$-valued locally constant
 functions $\phi$ on $G$ such that
\begin{itemize}
  \item $\phi(hg)=h.\phi(g)$, for all $h\in H, \,g\in G$; and
  \item   $\phi$ has compact support modulo (the left translations of )
  $H$.
\end{itemize}
It is a smooth representation of $G$ under right translations. The
following lemma is well known and easy to check.

\begin{lem}\label{tensorg0}
Let $V$ be a smooth representation of $G$. Then the linear map \be
\label{homf0}
   \begin{array}{rcl}
  V\otimes \ind_H^G V_0 &\rightarrow& \ind_H^G({V}|_H\otimes
  V_0),\\
   v\otimes \phi&\mapsto&(g\mapsto g.v\otimes \phi(g))
   \end{array}
\ee
 is a well defined isomorphism of smooth representations of $G$.

\end{lem}

\subsection{Generalized extensions}

Denote by $\CM(G)$ the category of smooth representations of $G$ (the
morphisms of this category are $G$-intertwining linear maps). By a
projective smooth representation of $G$, we mean a projective object
of the category $\CM(G)$.

\begin{lem}\label{tensorproj}
Let $V_1, V_2$ be two smooth representations of $G$. If  $V_1$
or $V_2$ is projective, then $V_1\otimes V_2$ is projective.
\end{lem}
\begin{proof}
This is well known. We sketch a proof for the convenience of the
reader. Without loss of generality, assume that $V_2$ is projective.
Note that $V_2$ is isomorphic to a quotient of $\ind_{\{1\}}^G V_2$.
Since it is projective, it is isomorphic to a direct summand of
$\ind_{\{1\}}^G V_2$. Therefore $V_1\otimes V_2$ is isomorphic to a
direct summand of
\[
 V_1\otimes  (\ind_{\{1\}}^G V_2)\cong \ind_{\{1\}}^G (V_1\otimes
 V_2)  \qquad(\textrm{by Lemma \ref{tensorg0}}).
\]
By \cite[Theorem A.4]{Ca}, $\ind_{\{1\}}^G (V_1\otimes
 V_2) $ is projective. Therefore $V_1\otimes V_2$ is also
 projective.
\end{proof}

\begin{lem}\label{gin}
Let $V_1$ and $V_2$ be two smooth representations of $G$. Let
$P_\bullet\rightarrow V_1$ be a projective resolution of $V_1$, and
let $V_2\rightarrow I^\bullet$ be an injective resolution of $V_2$.
Then for each $i\in \Z$, the $i$-th cohomology  of the complex
$\Hom_{G,k}(P_\bullet, V_2)$ and the  $i$-th cohomology  of the
complex $\Hom_{G,k}(V_1, I^\bullet)$ are both canonically isomorphic
to
\[
\left\{
  \begin{array}{ll}
    \Ext_G^i(\oJ_{G,k}\otimes V_1, V_2), & \hbox{for $k\in \BN$;} \\
   {\varinjlim}_r\Ext_G^i(\oJ_{G,r}\otimes V_1, V_2), & \hbox{for $k=\infty$.}
  \end{array}
\right.
\]
\end{lem}

\begin{proof}
First we assume that $k\in \BN$. Then \be \label{res0}
\Hom_{G,k}(P_\bullet, V_2)=\Hom_{G}(\oJ_{G,k}\otimes P_\bullet,
V_2), \qquad(\textrm{by Lemma \ref{ginv2}}). \ee By Lemma
\ref{tensorproj}, $\oJ_{G,k}\otimes P_\bullet\rightarrow
\oJ_{G,k}\otimes V_1$ is a projective resolution of
$\oJ_{G,k}\otimes V_1$. Therefore the $i$-th cohomology of the
complex  \eqref{res0} is canonically isomorphic to
$\Ext_G^i(\oJ_{G,k}\otimes V_1, V_2)$. On the other hand, it is
obvious that the $i$-th cohomology of the complex \be \label{res1}
\Hom_{G,k}(V_1, I^\bullet)=\Hom_{G,k}(\oJ_{G,k}\otimes V_1,
I^\bullet) \ee
 is canonically isomorphic to $\Ext_G^i(\oJ_{G,k}\otimes V_1, V_2)$.

The Lemma for $k=\infty$ then follows since taking cohomology
commutes with taking direct limits.
\end{proof}

Denote by $\CM_\mathrm u(\Lambda_G)$ the category of all locally
unipotent representations of $\Lambda_G$. For each $k=0,1,2,\cdots,
\infty$, we have a bi-functor
\[
  \Hom_{G, k}(\cdot\,,\cdot): \CM(G)^{\mathrm{op}}\times \CM(G)\rightarrow \CM_\mathrm u(\Lambda_G).
\]
In view of Lemma \ref{gin}, write
\[
  \Ext_{G, k}^i(\cdot\,,\cdot): \CM(G)^{\mathrm{op}}\times \CM(G)\rightarrow \CM_\mathrm u(\Lambda_G)
\]
for its $i$-th left derived bi-functor ($i\in \Z$).

Let $\Gamma$ be a directed set, i.e. $\Gamma$ is a partially ordered set with a partial order $\leq $ and  for any $\gamma,\gamma'\in \Gamma$, there exists $\gamma''\in \Gamma$, such that $\gamma\leq \gamma''$ and $\gamma'\leq \gamma''$.  We can view $\Gamma$ as a category where morphisms come from the partial order. Let $\Gamma^{\rm o}$ be the opposite category of $\Gamma$.
Let $\CC$ be an abelian category.
  A directed (resp. directed inverse ) system of objects in $\CC$ is a functor from $\Gamma$ (resp. $\Gamma^{\rm o}$) to $\CC$. We can write such a system as $\{V_\gamma\}_{\gamma\in \Gamma}$, where $V_\gamma\in \CC$ and for any $\gamma\leq \gamma'$ in $\Gamma$ we associate a morphism $\phi_{\gamma\gamma'}: V_\gamma\to V_{\gamma'}$ (resp.  $\phi_{\gamma\gamma'}: V_{\gamma'} \to V_{\gamma}$).  We call a directed (directed inverse) system $\{V_\gamma\}_{\gamma\in \Gamma}$ injective (resp. surjective ) if for any $\gamma\leq \gamma'$ the morphism $\phi_{\gamma\gamma'}$ is injective (resp. surjective).

\begin{lem}
\label{Ext_limit} Let $V$ be a smooth representation of $G$, and let
$\{V_\gamma\}_{\gamma\in \Gamma}$ be an injective  directed system of smooth
representations of $G$ where $\Gamma$ is a countable directed set.
Let $k\in \BN$. If  for all $i\in \Z$ and $\gamma\in \Gamma$,   $ \Ext_{G, k}^i(V_\gamma,  V)=0$, then  $\Ext_{G, k}^i(\varinjlim_{\gamma} V_\gamma, V)=0$ for all $i\in\Z$.
\end{lem}
\begin{proof}
In view of Lemma \ref{gin},   $\Ext_{G, k}^i(\varinjlim_{\gamma} V_\gamma, V)$ can be computed as $i$-th cohomology of
$\Hom_{G,k}(\varinjlim_{\gamma} V_\gamma, I^\bullet) $, where $I^\bullet=\{\cdots\to 0\to I^0\to I^1\to \cdots\}$ is an injective resolution of $V$. We have the following isomorphisms,
\begin{eqnarray*}
\Hom_{G,k}(\varinjlim_{\gamma} V_\gamma, I^\bullet)  &\simeq&  \Hom_G(\RJ_{G,k}\otimes (\varinjlim_{\gamma} V_\gamma), I^\bullet    ) \\
    &\simeq&  \varprojlim_{\gamma}  \Hom_G(\RJ_{G,k}\otimes V_\gamma, I^\bullet    )\\
    &\simeq&  \varprojlim_{\gamma}  \Hom_{G,k}( V_\gamma, I^\bullet    ),
    \end{eqnarray*}
    where the first and the third isomorphisms follow from Lemma \ref{ginv2}, and the second isomorphism is a general property of Hom functor.  Therefore it suffices to show that the inverse limit of the system of complexes $\{ \Hom_{G,k}( V_\gamma, I^\bullet    )\}_{\gamma\in\Gamma}$ is acyclic.

    Let $ X_\gamma^\bullet=\{\cdots\to 0\to X^0_\gamma\to  X^1_\gamma\to \cdots  \}$ be the cochain complex $\Hom_{G,k}( V_\gamma, I^\bullet    )$.  We get a directed inverse system of cochain complexes $\{X_\gamma^\bullet\}_{\gamma\in \Gamma}$. The directed inverse system $\{X_\gamma^i\}_{\gamma\in \Gamma}$ is surjective for each $i$ since $I^i$ is an injective module and $\phi_{\gamma\gamma'}:V_\gamma\to V_{\gamma'}$ is an injective morphism.   By assumption on the vanishing of $ \Ext_{G, k}^i(V_\gamma,  V)$ for any $i$ and $\gamma$,
      we get an acyclic complex of surjective directed inverse systems,
    $$\cdots\to 0\to  \{ X_\gamma^0 \}_{\gamma\in \Gamma}  \xrightarrow{ \{d_\gamma^0\}} \{ X^1_\gamma \}_{\gamma\in \Gamma}  \xrightarrow{\{d_\gamma^1\}}  \cdots .$$
 Let ${\rm Ker}_\gamma^i$ be the kernel of $d_{\gamma}^i$ and let ${\rm Im}_{\gamma}^i$ be the image of $d_{\gamma}^i$.
For every $i$, we have ${\rm Ker}_\gamma^i={\rm Im}_{\gamma}^{i-1}$.  Note that ${\rm Ker}_\gamma^1=X_\gamma^0$ and  we have short exact sequences
  $$0\to  {\rm Ker}_\gamma^i\to  X_\gamma^i\to   {\rm Ker}_\gamma^{i+1}\to 0.$$
 By induction it is easy to see that for all $i$ the directed inverse system $ \{{\rm Ker}_\gamma^i\}_{\gamma\in \Gamma}$ is surjective. Hence  for all $i$ we have the following short exact sequences ( see \cite[Lemma 10.85.4]{DJ})
   $$0\to  \varprojlim_{\gamma}    {\rm Ker}_\gamma^i\to  \varprojlim_{\gamma}   X_\gamma^i\to  \varprojlim_{\gamma}    {\rm Ker}_\gamma^{i+1}\to 0.$$
Combining all these short exact sequences, we conclude that the complex $\varprojlim_{\gamma}  X_\gamma^\bullet$ is acyclic.

\end{proof}

\subsection{A vanishing Theorem of generalized extensions}
The main result of this subsection is the following theorem.

\begin{thml}
\label{Vanishing_Theorem} Assume that $G=\mathsf G(\rk)$ for some connected
linear algebraic group $\mathsf G$ defined over $\rk$. Let $V_1$
and $V_2$ be two smooth representations of $G$. Assume that there
are two distinct characters $\chi_1$ and $\chi_2$ of $G$ such that
both $V_1\otimes \chi_1^{-1}$ and $V_2\otimes \chi_2^{-1}$ are
locally unipotent as representations of $G$, then
\[
  \Ext_{G, k}^i(V_1, V_2)=0,\qquad i\in \Z, \,k=0,1,2,\cdots, \infty.
\]

\end{thml}

We remark that Theorem \ref{Vanishing_Theorem} fails without the connectedness assumption on $\mathsf G$. Instead, we will use the following corollary in the disconnected case.

\begin{corl}
\label{corvanish} Let $G$ be an $\ell$-group which contains $\mathsf G(\rk)$ as an open normal subgroup of finite index, where $\mathsf G$ is a  connected linear algebraic group defined over $\rk$. Let $V_1$
and $V_2$ be two smooth representations of $G$. Assume that there
are two distinct characters $\chi_1$ and $\chi_2$ of $\mathsf G(\rk)$ such that
both $(V_1)|_{\mathsf G(\rk)}\otimes \chi_1^{-1}$ and $(V_2)|_{\mathsf G(\rk)}\otimes \chi_2^{-1}$ are
locally unipotent as representations of $\mathsf G(\rk)$, then
\[
  \Ext_{G, k}^i(V_1, V_2)=0,\qquad i\in \Z, \,k=0,1,2,\cdots, \infty.
\]

\end{corl}
\begin{proof}
Note that the
tensor product of two locally unipotent representations is also a
locally unipotent representation. By Lemma \ref{gin}, it suffices to
prove the corollary for $k=0$.
Let $P_\bullet$ be a projective resolution of $V_1$. Then $(P_\bullet)|_{\mathsf G(\rk)}$ is a projective resolution of $(V_1)|_{\mathsf G(\rk)}$. By Theorem \ref{Vanishing_Theorem}, the complex $\Hom_{\mathsf G(\rk)}(P_\bullet, V_2)$ is acyclic. Therefore the complex $\Hom_{G}(P_\bullet, V_2)$, which equals the complex $(\Hom_{\mathsf G(\rk)}(P_\bullet, V_2))^{G/\mathsf G(\rk)}$ of the $G/\mathsf G(\rk)$-invariant vectors, is also acyclic. This proves the corollary.
\end{proof}

The rest of this  subsection is devoted to a proof of Theorem \ref{Vanishing_Theorem}.

\begin{lem}
\label{translate} Let $V_1$ and $V_2$ be two smooth representations
of an $\ell$-group $G$. Then for each character $\chi$ of $G$, there
is an isomorphism
\[
  \Ext_{G, k}^i(V_1, V_2)\cong \Ext_{G, k}^i(V_1\otimes \chi, V_2\otimes \chi),\qquad i\in \Z, \,k=0,1,2,\cdots, \infty
\]
of locally unipotent representations of $\Lambda_G$.
\end{lem}

\begin{proof}
Take an injective resolution
\[
  0\rightarrow V_2 \rightarrow I_0\rightarrow I_1\rightarrow
  I_2\rightarrow \cdots
\]
of $V_2$. Then
\[
  0\rightarrow V_2\otimes \chi \rightarrow I_0\otimes \chi\rightarrow I_1\otimes \chi\rightarrow
  I_2\otimes \chi\rightarrow \cdots
\]
is  an injective resolution of $V_2\otimes \chi$. Therefore the
lemma follows.
\end{proof}

The following Lemma is well known and is an easy consequence of Lemma
\ref{tensorproj}.

\begin{lem}
\label{Ext_Coinvariant} Let $V_1, V_2$ be two smooth representations
of an $\ell$-group $G$. Then for all $i\in \Z$,
\[
 \Ext_G^i(V_1,V_2^\vee)\cong \oH_i(G, V_1\otimes V_2)^*.
\]
In particular,
$$\oH_i(G,V)^*\cong \Ext_G^i(V,\C), $$
for all smooth representation $V$ of $G$.
\end{lem}

Here and henceforth, a superscript ``$^\vee$" indicates the smooth
contragredient of a smooth representation,  a superscript ``$^*$"
indicates the space of all linear functionals, and ``$\oH_i$" indicates the $i$-th homology group.

\begin{lem}
\label{Vanishing_unip} Let $\mathsf U$ be a unipotent linear
algebraic group over $\rk$, and put $U:=\mathsf U(\rk)$. Let
$\chi$ be a character of $U$. Then for each $i\in \Z$,
$$\oH_i(U, \chi)=0 \quad \textrm{if $i\neq 0$ or $\chi$ is non-trivial.}$$
\end{lem}

\begin{proof}
By \cite[Proposition 10, Section 3.3]{Be}, the coinvariant functor
$V\mapsto V_U$ from the category $\CM(U)$  to the category of
complex vector spaces is exact. This implies the lemma.
\end{proof}
Similar to Lemma \ref{Vanishing_unip}, we have the following lemma for semisimple groups.
\begin{lem}
\label{Vanishing_ss} Let $\mathsf S$ be a connected semisimple linear
algebraic group over $\rk$, and put $S:=\mathsf S(\rk)$. Let
$\chi$ be a character of $S$. Then for each $i\in \Z$,
$$\oH_i(S, \chi)=0 \quad \textrm{if  $i\neq 0$ or $\chi$ is non-trivial.}$$
\end{lem}

\begin{proof}

In view of Lemma \ref{Ext_Coinvariant}, this is implied by \cite[Theorem A.13]{Ca}.
\end{proof}

The following lemma is a variant of Hoschild-Serre spectral
sequence, see \cite[Proposition A.9]{Ca}.
\begin{lem}
\label{Serre-Hoschild} Let  $H$ be a closed normal  subgroup of an $\ell$-group $G$.
Let $V$ and $W$ be smooth representations of $G$, with $H$ acting
trivially on $W$. Then there is a spectral sequence
$$E_2^{p,   q }= \Ext^p_{G/H}(\oH_q(H,V), W)\Rightarrow \Ext^{p+q}_G(V,W). $$
\end{lem}

Generalizing  Lemma \ref{Vanishing_ss}, we have the following lemma for reductive groups.
\begin{lem}
\label{Vanishing_red} Let $\mathsf L$ be a connected reductive linear
algebraic group over $\rk$, and put $L:=\mathsf L(\rk)$. Let
$\chi$ be a character of $L^\circ$. Then for each $i\in \Z$,
$$\oH_i(L^\circ, \chi)=0 \quad \textrm{if $i\neq 0$ or $\chi$ is non-trivial.}$$
\end{lem}

\begin{proof}
Write $\mathsf S$ for the derived subgroup of $\mathsf L$, and put $S:=\mathsf S(\rk)$.  Lemma \ref{Ext_Coinvariant} and Lemma \ref{Serre-Hoschild} imply that
\be\label{ohiext}
  \oH_i(L^\circ, \chi)^*\cong \Ext_{L^\circ/S}^i(\oH_0(S, \chi), \C).
\ee
If $i\neq 0$, then the right hand side of \eqref{ohiext} vanishes since $L^\circ/S$ is compact. The lemma is obvious for $i=0$.
\end{proof}

\begin{lem}\label{extdisj}
Let $(X, \CO_X)$ be a ringed space. Let $\CF_1$ and $\CF_2$ be two
$\CO_X$-modules. If the supports of $\CF_1$ and $\CF_2$ are
disjoint, then
 \[
  \Ext_{\CO_X}^i(\CF_1, \CF_2)=0,\qquad i\in \Z.
 \]
\end{lem}
\begin{proof}
By the construction of  injective resolutions as in
\cite[Chapter III, Proposition 2.2]{Ha}, we know that there is an
injective resolution $\CF_2\to \CI_\bullet$ such that the support of
$\CI_i$ is contained in that of $\CF_2$ ($i\in \Z$). Therefore the
lemma follows.
\end{proof}

\begin{lem}\label{extl}
Let $\Lambda$ be a finitely generated free abelian group. Let $V_1$
and $V_2$ be two representations of $\Lambda$. Assume that there are
two distinct characters $\chi_1$ and $\chi_2$ of $\Lambda$ such
that both $V_1\otimes \chi_1^{-1}$ and $V_2\otimes \chi_2^{-1}$ are
locally unipotent as representations of $\Lambda$, then
\[
\Ext_\Lambda^i(V_1, V_2)=0,\qquad i\in \Z.
 \]
\end{lem}
\begin{proof}
Write $\C[\Lambda]$ for the complex group algebra attached to
$\Lambda$. Then both $V_1$ and $V_2$ are $\C[\Lambda]$-modules, and
we have that
\[
\Ext_\Lambda^i(V_1, V_2)=\Ext_{\C[\Lambda]}^i(V_1, V_2),\qquad i\in
\Z.
 \]
Denote by $\widetilde{\C[\Lambda]}$ the structure sheaf of the
scheme $\mathrm{Spec}(\C[\Lambda])$, and  denote by $\widetilde V_1$
and $\widetilde V_2$ the quasi-coherent
$\widetilde{\C[\Lambda]}$-modules attached to $V_1$ and $V_2$,
respectively.  Note that there exists a filtration $0=V_1^0\subset V_1^1\subset V_1^2\subset \cdots  $ of the representation $V_1$ of $\Lambda$  such that
$\bigcup_{k\geq 1} V_1^k=V_1$ and $V_1^{k}/V^{k-1}_1$ is a direct sum of copies of $\chi_1$  for every $k\geq 1$.

Using Lemma \ref{Ext_limit}, we are reduced to show that for all $k$
$$\Ext_\Lambda^i(V_1^k, V_2)=0,  \qquad i\in
\Z.
 $$
 For any vector space $W$,  we always have
 $$\Ext^i_{\Lambda}(W\otimes \chi_1,V_2)=W^*\otimes \Ext^i_{\Lambda}(\chi_1,V_2)$$
  for every $i$, where $W^*$ is the dual vector space of $W$.   By induction on $k$, we  assume without loss of generality that $V_1=\chi_1$. Then we have that (see
\cite[Chapter III, exercise 6.7]{Ha})
 \[
\Ext_{\C[\Lambda]}^i(V_1,
V_2)=\Ext_{\widetilde{\C[\Lambda]}}^i(\widetilde V_1, \widetilde
V_2),\qquad i\in \Z.
 \]
 Since $\chi_1\neq \chi_2$, the supports of  $\widetilde V_1$ and $\widetilde V_2$ are disjoint. Therefore the lemma follows by Lemma \ref{extdisj}.
\end{proof}

\begin{lem}
\label{Vanishing_red2} Let $\mathsf L$ be a connected reductive linear
algebraic group over $\rk$, and put $L:=\mathsf L(\rk)$. Let
$\chi$ be a non-trivial character of $L$, and let $V$ be a locally unipotent representation of $L$. Then
$$\Ext_L^i(\chi, V)=0, \quad \textrm{for all $i\in \Z$.}$$
\end{lem}

\begin{proof}
Note that $L^\circ$ acts trivially on $V$.  Lemma \ref{Serre-Hoschild} and  Lemma \ref{Vanishing_red} imply that
$$
\Ext_L^i(\chi, V)\cong\Ext_{L/L^\circ}^i(\oH_0(L^\circ, \chi), V), \quad \textrm{for all $i\in \Z$.}
$$
If $\chi|_{L^\circ}$ is non-trivial, then the above space vanishes. Now assume that $\chi|_{L^\circ}$ is trivial. Then $\oH_0(L^\circ, \chi)$ is a non-trivial one-dimensional representation of $L/L^\circ$. The lemma then follows by Lemma \ref{extl}.
\end{proof}

Now we come to the proof of Theorem \ref{Vanishing_Theorem}.
Using
Lemma \ref{translate}, we assume without loss of generality that
$\chi_2$ is trivial. Then $\chi_1$ is non-trivial. As in the proof of Corollary \ref{corvanish}, it suffices to
prove Theorem \ref{Vanishing_Theorem} for $k=0$. As in the proof of Lemma \ref{extl}, we may use Lemma
\ref{Ext_limit} to further assume that $V_1=\chi_1$.
Then what we need to prove is that \be\label{extgi0}
  \Ext_G^i(\chi_1, V_2)=0,\qquad i\in \Z
\ee for all non-trivial character $\chi_1$ of $G$, and all locally
unipotent  representation $V_2$ of $G$.

Denote by $\mathsf N$ the unipotent radical of $\mathsf G$, and
put $N:=\mathsf N(\rk)$. Note that $N$ acts trivially on $V_2$
since $N\subset G^\circ$. If  $\chi_1$ is non-trivial on $N$, then
Lemma \ref{Serre-Hoschild} and Lemma \ref{Vanishing_unip} imply that
\eqref{extgi0} holds. Now assume that $\chi_1$ is trivial on $N$.
Then   Lemma \ref{Serre-Hoschild} and Lemma \ref{Vanishing_unip}
imply that
\[
  \Ext_G^i(\chi_1, V_2)\cong\Ext_{G/N}^i(\chi_1, V_2),\qquad i\in \Z,
\]
which vanishes by Lemma \ref{Vanishing_red2}. This finishes the proof  of Theorem \ref{Vanishing_Theorem}.

\section{A localization principle for extensions}

\label{loc_princ}

\subsection{Equivariant $\ell$-sheaves and the localization principle}
Let $X$ be an $\ell$-space. We define an $\ell$-sheaf on $X$ to be a
sheaf of complex vector spaces on $X$. For any $\ell$-sheaf $\CF$ on
$X$, let $\Gamma_c(\CF)$ denote the space of all  global sections of
$\CF$ with compact support. In particular,
$\oS(X)=\Gamma_c(\C_X)$, where $\C_X$ denotes the sheaf of
locally constant $\C$-valued functions on $X$. For each $x\in X$,
denote by $\CF_x$ the stalk of $\CF$ at $x$; and for each $s\in
\Gamma_c(\CF)$, denote by $s_x\in \CF_x$ the germ of $s$ at $x$.
The  set $\bigsqcup_{x\in X} \CF_x$ carries a unique topology  such
that for all $s\in \Gamma_c(\CF)$, the map
  \[
    X\rightarrow \bigsqcup_{x\in X} \CF_x, \quad x\mapsto s_x
  \]
is an open embedding. Then $\Gamma_c(\CF)$ is naturally identified
with the space of all compactly supported continuous sections of the
map $\bigsqcup_{x\in X} \CF_x\rightarrow X$.

Let $G$ be an $\ell$-group which acts continuously on an
$\ell$-space $X$.

\begin{dfn}(\cf \cite[Section 1.17]{BZ})
A $G$-equivariant $\ell$-sheaf on $X$ is an $\ell$-sheaf $\CF$ on
$X$, together with a continuous group action
\[
  G\times \bigsqcup_{x\in X} \CF_x\rightarrow \bigsqcup_{x\in X} \CF_x
\]
such that for all $x\in X$, the action of each $g\in G$ restricts to a
linear map $\CF_x\to \CF_{g.x}$.
\end{dfn}

 Given a $G$-equivariant $\ell$-sheaf $\CF$ on $X$, the space $\Gamma_c(\CF)$ is a smooth representation of $G$ so that
\[
  (g.s)_{g.x}=g.s_x \quad \textrm{for all } g\in G,\, x\in X,  \,s\in \Gamma_c(\CF).
\]
For each $G$-stable locally closed subset $Z$ of $X$, the
restriction $\CF|_Z$ is clearly a  $G$-equivariant $\ell$-sheaf on
$Z$.

The main purpose of this section is to prove the following
localization principle for extensions.
\begin{thmd}
\label{Fiber_Vanishing} Let $\CF$ be a $G$-equivariant $\ell$-sheaf
on $X$. Let $Y$ be an $\ell$-space with a continuous map $\pi:X\to
Y$ so that $\pi(g.x)=\pi(x)$, for all $x\in X$ and $g\in G$. Let
$V_1$, $V_2$ be two smooth representations of $G$, and let $i\in
\Z$. Assume  that
$$\Ext^i_{G}(\Gamma_c(\CF|_{X_y})\otimes V_1,V_2^\vee)=0 \qquad \text{ for all }  y\in Y,$$
where $X_y:=\pi^{-1}(y)$. Then
$$ \Ext_{G}^i(\Gamma_c(\CF)\otimes V_1, V_2^\vee)=0.
$$
\end{thmd}

By Lemma \ref{gin}, Theorem
\ref{Fiber_Vanishing} has the following obvious consequence.
\begin{cord}
\label{Fiber_Vanishing2} Let $\CF$ and  $\pi:X\to Y$ be as in
Theorem \ref{Fiber_Vanishing}. Let $\chi$ be a character of $G$. Let
$k\in \BN, $ and let $i\in \Z$. Assume  that
$$\Ext^i_{G,k}(\Gamma_c(\CF|_{X_y}),\chi)=0 \qquad \text{ for all }  y\in Y,$$
where $X_y:=\pi^{-1}(y)$. Then
$$ \Ext_{G,k}^i(\Gamma_c(\CF), \chi)=0.
$$
\end{cord}

\subsection{A projective generator}

Write $\oH(G)$ for the Hecke algebra of $G$, namely
$\oH(G):=\oS(G)\,dg$, for a left invariant Haar measure
$dg$ on $G$. Denote by $\C_{G,X}$ the sheaf of $\oH(G)$-valued
locally constant functions on $X$. It is a $G$-equivariant
$\ell$-sheave under the diagonal action of $G$ on $\oH(G)\times X$.
Here $G$ acts on $\oH(G)$ by the left translations, and the obvious
identification
\[
   \bigsqcup_{x\in X} (\C_{G,X})_x=\oH(G)\times X
\]
is used.

Denote by $\CS h_G(X)$ the abelian category of $G$-equivariant
$\ell$-sheaves on $X$ (a morphism in this category is a sheaf
homomorphism $\CF\rightarrow \CF'$ so that the induced map
$\bigsqcup_{x\in X} \CF_x\rightarrow \bigsqcup_{x\in X} \CF'_x$ is
$G$-equivariant). Denote by $\CM_X(G)$ the category of all smooth
representations $V$ of $G$ equipped with a non-degenerate
$\oS(X)$-module structure on it such that
\[
 g.(\phi v)=(g.\phi)(g. v),\quad\textrm{for all } g\in G,\, \phi\in \oS(X), \,v\in V.
\]
Here the $\oS(X)$-module structure is  non-degenerate
means that
\[
   \oS(X)\cdot V=V.
\]
By \cite[Proposition 1.14]{BZ}, $\Gamma_c$ establishes an
equivalence between the category of $\ell$-sheaves on $X$ and the
category of non-degenerate $\oS(X)$-modules. This implies the following equivariant version of the localization theorem.
\begin{proposition}
\label{equi-loc}
The functor \be \label{equgc}
  \Gamma_c: \CS h_G(X)\rightarrow \CM_X(G)
\ee is an equivalence of categories.
\end{proposition}

\begin{lemd}\label{projx}
For each $\ell$-space $X$, the sheaf  $\C_X$ is a projective object
in the category of $\ell$-sheaves on $X$.
\end{lemd}
\begin{proof}
By the equivalence of categories, we only need to show that
$\oS(X)$ is a projective object in the category of
non-degenerate $\oS(X)$-modules. It is elementary and
well known that $X$ is a countable disjoint union of open compact subsets:
\[
  X=\bigsqcup_{i\in I} X_i.
\]
Then $\oS(X)=\bigoplus_{i\in I} \oS(X_i)$ and
the lemma easily follows.
\end{proof}

When $X$ has only one element, the following proposition is proved
by P. Blanc \cite{Bla}. See also \cite[Theorem A.4]{Ca}.
\begin{prpd}
\label{Generator_Equivariant} The  $G$-equivariant $\ell$-sheaf
$\C_{G,X}$ is a projective generator in $\CS h_G(X)$, that is, it is
a projective object of $\CS h_G(X)$, and for each $G$-equivariant
$\ell$-sheaf $\CF$ on $X$, there exist an epimorphism
$\bigoplus_{i\in I} \C_{G,X}\to \CF$ in $\CS h_G(X)$ for some index
set $I$.
\end{prpd}
\begin{proof}
By Proposition \ref{equi-loc}, we only need to show that
$\oS(X) \otimes \oH(G)$ is a projective generator in
$\CM_X(G)$. Here $G$ acts on $\oS(X) \otimes \oH(G)$
diagonally, and $\oS(X)$ acts on $\oS(X) \otimes
\oH(G)$ through the multiplication on $\oS(X)$.

For every $V\in \CM_X(G)$, the linear map
\[
   \oS(X) \otimes \oH(G)\otimes V\to V,\quad \phi\otimes \eta\otimes v\mapsto \phi\cdot(\eta. v),
\]
is an epimorphism in  $\CM_X(G)$, where $G$ and $\oS(X)$
act on $\oS(X) \otimes \oH(G)\otimes V$ through their
action on $\oS(X) \otimes \oH(G)$. This proves the second
assertion of the proposition.

Now we show that $\oS(X) \otimes \oH(G)$ is projective.  Fix an
element $\eta_0\in \oS(G)$ so that
\[
  \int_G \eta_0(g)\,d_r g=1,
\]
where $d_r g$ denotes a fixed right invariant Haar measure $G$.

Let
\[
\xymatrix{
 &\oS(X)\otimes \oH(G)\ar[d]^{F} &\\
 U \ar[r]^{P} &V \ar[r] & 0
}
\]
be a diagram  in $\CM_X(G)$ so that the map $P$ is surjective. Lemma
\ref{projx} implies that there exists a $\oS(X)$-module
homomorphism $F': \oS(X)\otimes \oH(G)\rightarrow U$ which
is a lifting of  $F$, that is, $P\circ F'=F$. Define a linear map
\[
  F'': \oS(X)\otimes \oH(G)\rightarrow U, \quad \phi\otimes \omega\mapsto \int_G g^{-1}.F'(g.\phi\otimes (\eta_0\cdot(g.\omega)))\,d_r g.
\]
Then it is routine to check that $F''$ is a well-define morphism in
$\CM_X(G)$ which lifts $F$. This finishes the proof.
\end{proof}

\begin{cord}\label{gcexact}
The functor $\Gamma_c: \CS h_G(X)\rightarrow \CM(G)$ is exact and
maps projective objects to projective objects.
\end{cord}
\begin{proof}
The functor is exact since \eqref{equgc} is an equivalence of
categories. Proposition \ref{Generator_Equivariant} implies that
every projective object in $\CS h_G(X)$ is isomorphic to a direct
summand  of $\bigoplus_{i\in I} \C_{G,X}$ for some index set $I$.
Lemma \ref{tensorg0} implies that as representations of $G$,
$\Gamma_c(\C_{G,X})=\oS(X)\otimes \oH(G)$ is isomorphic to
a direct sum of copies of $\oH(G)$. As a special case of Proposition
\ref{Generator_Equivariant}, we know that $\oH(G)$ is a projective
object in $\CM(G)$. This proves that the functor $\Gamma_c$  maps
projective objects to projective objects.
\end{proof}

\begin{cord}\label{gcexact2}
Let $Z\subset X$ be a $G$-stable locally closed subset of $X$. Then
the functor
\[
 \CS h_G(X)\rightarrow \CS h_G(Z), \quad \CF\mapsto \CF|_Z
 \]
  is exact and maps projective objects to projective objects.
\end{cord}

\begin{proof}
Since $\C_{G,X}|_Z=\C_{G,Z}$, the corollary follows by the argument
as in the proof of Corollary \ref{gcexact}.
\end{proof}

\subsection{The proof of Theorem \ref{Fiber_Vanishing}}

Let $\CF$ be a $G$-equivariant $\ell$-sheaf on $X$ as in Theorem
\ref{Fiber_Vanishing}. For each smooth representation $V$ of $G$,
$\CF\otimes V$ is clearly a $G$-equivariant $\ell$-sheaf on $X$.
Moreover, we have that \be\label{gtensor}
  \Gamma_c(\CF\otimes V)=\Gamma_c(\CF)\otimes V
\ee as a smooth representation of $G$.

Let  $Y$ and $\pi: X\rightarrow Y$ be as in Theorem
\ref{Fiber_Vanishing}. Note that $\Gamma_c(\CF)$ is a
$\con^\infty(X)$-module, where $\con^\infty(X)$ denotes the algebra
of all $\C$-valued locally constant functions on $X$. The pull-back
through $\pi$ yields an algebra homomorphism $\oS(Y)\to
\con^\infty(X)$. Using this homomorphism, we view  $\Gamma_c(\CF)$
as a non-degenerate $\oS(Y)$-module. For each smooth
representation $V$ of $G$, recall that its co-invariant space is
defined to be
\[
  V_G:=\frac{V}{\span \{g.v-v\mid g\in G, \,v\in V\}}.
\]
For each non-degenerate $\oS(Y)$-module $M$ and each $y\in
Y$, denote by $M_y$ the stalk at $y$ of the $\ell$-sheaf $\tilde M$ on $Y$
associated to $M$. To be explicit,
\[
  M_y:=M\otimes_{\oS(Y)} \C_{y},
\]
where $\C_y$ denotes the ring $\C$ with the evaluation map $\oS(Y)\to \C$  at $y$.

The following proposition is proved in \cite[Proposition 2.36]{BZ}.
\begin{prpd}\label{bloc}
The coinvariant space $(\Gamma_c(\CF))_G$ is a non-degenerate $\oS(Y)$-module. Moreover, for
each $y\in Y$, one has a natural vector space isomorphism
\[
  ((\Gamma_c(\CF))_G)_y\cong (\Gamma_c(\CF|_{X_y}))_G\qquad (X_y:=\pi^{-1}(y)).
\]
\end{prpd}

Now we come to the proof of Theorem \ref{Fiber_Vanishing}.
 In view
of Lemma \ref{Ext_Coinvariant} and the equality \eqref{gtensor},
replacing $\CF$ by $\CF\otimes V_1\otimes V_2$, we only need to show
that \be\label{vohi}
 \oH_i(G,\Gamma_c(\CF))=0,
\ee under the assumption that \be\label{ass0}
\oH_i(G,\Gamma_c(\CF|_{X_y}))=0 \qquad \text{ for all }  y\in Y. \ee

Take a projective resolution $\CP_\bullet\rightarrow \CF$ of $\CF$
in the category $\CS h_G(X)$. By Corollary \ref{gcexact} and
Corollary \ref{gcexact2}, for all $y\in Y$,
$\Gamma_c({\CP_\bullet}|_{X_y})\rightarrow \Gamma_c(\CF|_{X_y})$ is
a projective resolution of $\Gamma_c(\CF|_{X_y})$ in the category
$\CM(G)$. By the assumption of \eqref{ass0}, the complex
$(\Gamma_c({\CP_\bullet}|_{X_y}))_G$ is exact at degree $i$.
Applying Proposition \ref{bloc}, we know that the complex
$((\Gamma_c({\CP_\bullet}))_G)_y$ is exact at degree $i$. Therefore
$(\Gamma_c({\CP_\bullet}))_G$ is exact at degree $i$ as a complex of
non-degenerate $\oS(Y)$-modules. (Recall that the category
of $\ell$-sheaves on $Y$ is equivalent to the category of
non-degenerate $\oS(Y)$-modules.) This proves \eqref{vohi}
since by Corollary
 \ref{gcexact}, $\Gamma_c(\CP_\bullet)\rightarrow \Gamma_c(\CF)$ is a projective resolution of $\Gamma_c(\CF)$ in the category $\CM(G)$.

\section{A theorem of automatic extensions}\label{secaut}

\subsection{Frobenius reciprocity and Shapiro's lemma}\label{sfro}
Let $H$ be a closed subgroup of  an $\ell$-group $G$. Then
there is a unique character $\delta_{H\backslash G}$ of $H$ such
that \be\label{homdhg}
  \Hom_G(\ind_H^G \delta_{H\backslash G}^{-1},\C)\neq 0.
\ee Here $\ind_H^G$ indicate the un-normalized Schwartz induction as
in Section \ref{secshw}. The space \eqref{homdhg} is then
one-dimensional.

Let $V$ be a smooth representation of $G$ and let $V_0$ be a smooth
representation of $H$. Recall the following well-known Frobenius
reciprocity.
\begin{lem}\label{lemfro}
Fix a generator of the space \eqref{homdhg}. Then there is a
canonical linear isomorphism
\[
     \Hom_G(\ind_H^G  V_0,V^\vee)\cong  \Hom_H(\delta_{H\backslash G} \otimes V_0, (V|_H)^\vee).
\]
\end{lem}

Combining Lemmas \ref{ginv2},  \ref{tensorg0} and \ref{lemfro},  we
get the following proposition.

\begin{prpl}\label{gfrob}
Fix a generator of the space \eqref{homdhg}. Then there is a
canonical linear isomorphism
\[
   \Hom_{G,k}(\ind_H^G V_0,V^\vee)\cong \Hom_H(\delta_{H\backslash
   G} \otimes  \oJ_{G,k}\otimes V_0, (V|_H)^\vee).
\]
\end{prpl}

It is well known that Schwartz inductions preserve projectiveness, as in the following lemma.
\begin{lem}\label{tensorprojind}
If  $V_0$ is projective as a smooth representations of $H$, then the
smooth representation $\ind_H^G V_0$ of $G$ is also  projective.
\end{lem}
\begin{proof}
As in the proof of Lemma \ref{tensorproj}, $V_0$ is isomorphic to a
direct summand of $\ind_{\{1\}}^H V_0$. Therefore $\ind_H^G V_0$ is
isomorphic to a direct summand of
\[
  \ind_H^G (\ind_{\{1\}}^H V_0)\cong \ind_{\{1\}}^G V_0.
\]
By \cite[Theorem A.4]{Ca}, $\ind_{\{1\}}^G V_0 $ is projective.
Therefore $\ind_H^G V_0$ is also
 projective.
\end{proof}

We have the following Shapiro's lemma for generalized extensions.
\begin{prpl}\label{gshap}
Fix a generator of the space \eqref{homdhg}. Then there is a
canonical linear isomorphism
\[
   \Ext_{G,k}^i(\ind_H^G V_0,V^\vee)\cong \Ext_H^i(\delta_{H\backslash
   G} \otimes  \oJ_{G,k}\otimes V_0, (V|_H)^\vee), \quad k\in \BN, \,i\in \Z.
\]
\end{prpl}

\begin{proof}
Take a projective resolution $P_\bullet\rightarrow V_0$ of $V_0$.
Since ``$\ind$" is an exact functor, by Lemma \ref{tensorprojind},
$\ind_H^G P_\bullet\to \ind_H^G V_0$ is also a projective
resolution. Then $\Ext_{G,k}^i(\ind_H^G V_0,V^\vee)$ equals the
$i$-th cohomology of the complex
\[
   \Hom_{G,k}(\ind_H^G P_\bullet,V^\vee).
   \]
The later is isomorphic to the complex
\[
 \Hom_H(\delta_{H\backslash G} \otimes  \oJ_{G,k}\otimes P_\bullet,{V_H}^\vee)
 \]
 by Proposition \ref{gfrob}. By Lemma \ref{tensorproj}, $\delta_{H\backslash G} \otimes  \oJ_{G,k}\otimes P_\bullet\to \delta_{H\backslash
   G} \otimes  \oJ_{G,k}\otimes V_0$ is also a projective resolution in the category $\CM(H)$. Therefore the proposition follows.

\end{proof}

\subsection{The case of homogeneous spaces}
Let $\chi$ be a character of $G$.  Note that there exists a natural isomorphism $\oS(G/H)\simeq \ind_H^G \C$ as representations of $G$ via the following map
$$\phi \mapsto  \{g\mapsto \phi(g^{-1})\}, $$
for any $\phi\in \oS(G/H)$.  Hence
 as a special case of Proposition
\ref{gshap}, we have the following proposition.

\begin{prpl}\label{gshap2}
There is a linear isomorphism
\[
   \Ext_{G,k}^i(\oS(G/H),\chi)\cong \Ext_H^i(\delta_{H\backslash
   G} \otimes  \oJ_{G,k}, \chi), \quad k\in \BN; \,i\in \Z.
\]
\end{prpl}

Recall from the induction that $G/H$ is said to be $\chi$-admissible
if $\Hom_G(\oS(G/H),\chi)\neq 0$. Proposition \ref{gshap2}
implies that \be\label{ado}
  \textrm{$G/H$ is $\chi$-admissible $\ \Longleftrightarrow\ $ $\chi|_H =\delta_{H\backslash G}$.}
\ee

\begin{thml}\label{gshap3}
Assume that $H$ contains $\mathsf H(\rk)$ as an open normal subgroup of finite index, where $\mathsf H$ is a connected linear algebraic group defined over $\rk$. If
$G/\mathsf H(\rk)$ is not $\chi$-admissible, then
\[
   \Ext_{G,k}^i(\oS(G/H),\chi)=0, \quad k=0,1,2,\cdots, \infty;\,i\in \Z.
\]
\end{thml}
\begin{proof}
When $k$ is finite, this is a direct consequence of Proposition
\ref{gshap2} and Corollary \ref{corvanish}. Then for
$k=\infty$, the theorem follows by Lemma \ref{gin}.

\end{proof}

\subsection{The automatic extension theorem}

Let $\mathsf G$ be a linear algebraic group  over $\rk$, acting
algebraically on an algebraic variety $\mathsf Z$  over $\rk$. We
say that $\mathsf Z$ is homogeneous if the action of $\mathsf
G(\bar \rk)$ on $\mathsf Z(\bar \rk)$ is transitive, where $\bar
\rk$ denotes an algebraic closure of $\rk$. The following result on
homogeneous spaces over $p$-adic fields is  well known.

 \begin{lem}\label{finitehom}\cite[Section 6.4, Corollary 2 and Section 3.1, Corollary 2]{PR}
 If $\mathsf Z$ is homogeneous, then $\mathsf Z(\rk)$ has only finitely many $\mathsf G(\rk)$-orbits, and every  $\mathsf G(\rk)$-orbit is open in $\mathsf Z(\rk)$.
 \end{lem}

In general, recall the following result which is due to  M.
Rosenlicht \cite{Ro}. See also \cite[Theorem 4.4, p.187]{PV}.
\begin{prpl}
\label{Quotient_Rationality} There exists a $\mathsf G$-stable
open dense subvariety $\mathsf U$ of $\mathsf Z$, a variety
$\mathsf V$ over $\rk$, and a $\mathsf G$-invariant morphism $f:
\mathsf U\to \mathsf V$ of algebraic varieties over $\rk$ such
that for all $\rk$-rational point $y\in \mathsf V$, the subvariety
$f^{-1}(y)$ of $\mathsf U$ is homogeneous (under the action of
$\mathsf G$).
\end{prpl}

Let $\chi$ be a character of $\mathsf G(\rk)$ as in Theorem
\ref{main1}. Recall the notion of weakly $\chi$-admissible from Definition \ref{weakad}.

\begin{thml}
\label{vextz} Assume that every $\mathsf G(\rk)$-orbit in
$\mathsf Z(\rk)$ is not weakly $\chi$-admissible. Then
\[
   \Ext_{\mathsf G(\rk),k}^i(\oS(\mathsf Z(\rk)),\chi)=0, \quad\textrm{for all } k=0,1,2,\cdots, \infty; \,i\in \Z.
\]
\end{thml}
\begin{proof}
Using Lemma \ref{gin}, we assume that $k$ is finite. Using
Proposition \ref{Quotient_Rationality} inductively, and using the
long exact sequences for extensions, we assume without loss of
generality that $\mathsf Z$ equals the variety $\mathsf U$ of
Proposition \ref{Quotient_Rationality}. Then the morphism $f$ of
Proposition \ref{Quotient_Rationality} yields a $\mathsf
G(\rk)$-invariant continuous map
\[
   f_0: \mathsf U(\rk)\rightarrow \mathsf V(\rk).
\]
By Corollary \ref{Fiber_Vanishing2}, we only need to show that
\[
   \Ext_{\mathsf G(\rk),k}^i(\oS(f_0^{-1}(y)),\chi)=0.
\]
for all $y\in \mathsf V(\rk)$. This is obviously implied by Lemma
\ref{finitehom} and Theorem \ref{gshap3}.
\end{proof}

We remark that Theorem \ref{vextz} fails if  the condition ``not weakly $\chi$-admissible" is replaced  by the weaker condition ``not $\chi$-admissible", even when $\mathsf G$ is connected and reductive, and $\mathsf Z$ is $\mathsf G$-homogeneous.

As in Theorem \ref{main1}, let $\mathsf X$ be an algebraic variety
over $\rk$ on which $\mathsf G$ acts algebraically, and let
$\mathsf U$ be a $\mathsf G$-stable open subvariety of
$\mathsf X$. Using the long exact sequence for generalized
extensions, Theorem \ref{vextz} clearly implies the following
automatic extension theorem, which contains Theorem \ref{main1} as a
special case.
\begin{thml}\label{autoext}
 Assume that every $\mathsf G(\rk)$-orbit in $(\mathsf X\setminus \mathsf U)(\rk)$ is not weakly $\chi$-admissible. Then for every $k=0,1,2,\cdots, \infty$ and every $i\in \Z$, the restriction map
\[
  \Ext_{\mathsf G(\rk),k}^i(\oS(\mathsf X(\rk)),\chi)\to \Ext_{\mathsf G(\rk),k}^i(\oS(\mathsf U(\rk)),\chi)
\]
is a linear  isomorphism.
\end{thml}

\section{Semialgebraic  spaces and meromorphic continuations}
\label{sec_Igusa}
For the proof of Theorem \ref{main2}, we describe a general form of
the rationality of Igusa's zeta integral, in the setting of
semialgebraic geometry over $p$-adic fields.
For the basics of $p$-adic semialgebraic geometry, we refer the readers to the following papers \cite{CCL,Cl,ClL,De,DV,Ma}.
\subsection{Semialgebraic  spaces}
Recall that a subset of $\rk^n$ ($n\in \BN$) is said to be
semialgebraic if it is a finite Boolean combination of sets of the
form
\[
   \{x\in \rk^n\mid f(x)=y^k\textrm{ for some } y\in \rk^\times\},
\]
where $f: \rk^n\rightarrow \rk$ is a polynomial function, and $k$ is
a positive integer. Given a semialgebraic subset $X$ of $\rk^n$ and
$Y$ of $\rk^m$ ($m\in \BN$), a map from $X$ to $Y$ is said to be
semialgebraic if its graph is a semialgebraic subset of $\rk^{n+m}$.

Let $X$ be a set.  We denote by $\mathrm A_X$  the set of all  triples $(U,U', \phi)$, $U$ is a semialgebraic subset of $\rk^n$ for some $n\in \BN$, $U'$ is a subset of $X$, and $\phi:U\rightarrow U'$ is a bijection.

\begin{dfn}\label{dfnnash}
 A semialgebraic structure over $\rk$ on a set $X$ is a subset  $\mathcal A$ of  $\mathrm A_X$ with the following properties:
\begin{enumerate}
\item[(a)]
 every two elements $(U_1,U'_1, \phi_1)$ and $(U_2,U'_2,\phi_2)$ of $\mathcal A$ are
 semialgebraically  compatible, namely, the bijection
  \[
  \phi_2^{-1}\circ \phi_1: \phi_1^{-1}(U_1'\cap U_2')\rightarrow
 \phi_2^{-1}(U_1'\cap U_2')
 \]
has semialgebraic domain and codomain, and is semialgebraic;
\item[(b)]
  there are finitely many elements $(U_i,U'_i,\phi_i)$ of $\mathcal A$,
 $i=1,2,\cdots, r$ ($r\in \BN$), such that
\[
   X=U'_1\cup U'_2\cup \cdots \cup U_r';
\]
  \item[(c)]
  for every element of $\mathrm A_X$, if it is semialgebraically compatible with all elements of $\CA$, then itself is an element of $\CA$.
  \end{enumerate}
\end{dfn}

A semialgebraic  space over $\rk$ (or a semialgebraic space for
brevity) is defined to be a set together with a semialgebraic
structure (over $\rk$) on it. By a semialgebraic chart of a
semialgebraic space, we mean an element of the semialgebraic
structure.

The following lemma is routine to  check.
\begin{lemd}\label{defnashs0}
With the notation as in Definition \ref{dfnnash}, let
\[
 \mathcal A_0=\{( U_i, U_i', \phi_i)\mid i=1,2,\cdots,r\}
 \]
be a finite subset of $\RA_X$ whose elements are pairwise
semialgebraically compatible with each other. If
\[
  X=U'_1\cup U'_2\cup \cdots \cup U_r',
\]
then the set of all elements in $\RA_X$ which are semialgebraically
compatible with all elements of $\CA_0$ is a semialgebraic structure
on $X$.
\end{lemd}

By Lemma \ref{defnashs0}, it is clear that the product of two
semialgebraic spaces is naturally a semialgebraic space.

\begin{dfn}
A subset $S$ of a semialgebraic space $X$ is said to be
semialgebraic if $\phi^{-1}(S\cap U')$ is semialgebraic  for every
semialgebraic
 chart $(U,U',\phi)$ of $X$. A map from a  semialgebraic  space $X$ to a semialgebraic space $Y$ is said to be  semialgebraic if its graph is semialgebraic in $X\times Y$.

 \end{dfn}

It is clear that every semialgebraic subset of a semialgebraic space
is naturally a semialgebraic space.  Recall the following famous result
of Macintyre \cite{Ma}.

\begin{lemd}\label{tarski}
Let $f:X\rightarrow Y$ be a semialgebraic map of  semialgebraic
spaces. Then for each semialgebraic subset $S$ of $X$, $f(S)$ is a
semialgebraic subset of $Y$.
\end{lemd}
It is well-known and easy to see that Lemma \ref{tarski} implies
that the composition of two semialgebraic maps is semialgebraic, and
the inverse image of a  semialgebraic set under a semialgebraic map
is semialgebraic. All semialgebraic spaces form a category whose
morphisms are semialgebraic maps.

\begin{dfn}
 The dimension $\dim X$ of a semialgebraic space $X$ is defined to be the largest non-negative integer $n$ such that there is a semialgebraic subset of $X$ which is isomorphic to a non-empty open semialgebraic subset of $\rk^n$ as a semialgebraic space. By convention,  the dimension of the empty set is defined to be $-\infty$.
\end{dfn}

The following proposition asserts that infinite semialgebraic spaces are classified by their dimensions.

\begin{prpd}\cite[Theorem 2]{Cl}
Every infinite semialgebraic space $X$ has positive dimension and is
isomorphic to $\rk^{\dim X}$.
\end{prpd}

 By a semialgebraic function on a semialgebraic space, we mean a semialgebraic map from it to $\rk$.

\begin{dfn}
A $\C$-valued function on a semialgebraic space $X$ is said to be
definable of order $\leq 0$ if
 it belongs to the $\C$-algebra generated by the functions of the form
\[
  1_S,\quad \abs{f}_\rk^{s_0},
\]
where $s_0\in \C$, $1_S$ denotes the characteristic function of a
semialgebraic subset $S$ of $X$, and $f$ is a nowhere vanishing
semialgebraic function on $X$. It is  said to be definable of order
$\leq k$ ($k\geq 1$) if
 it is a linear combination of the functions of the form
\[
  \phi\cdot (\mathrm{val}\circ g_1)\cdot (\mathrm{val}\circ g_2)\cdot\cdots \cdot (\mathrm{val}\circ g_k),
\]
 where $\phi$ is a  definable function on $X$ of order $\leq 0$, and $g_1, g_2, \cdots, g_k$ are nowhere vanishing semialgebraic functions on $X$.
\end{dfn}

Here $\abs{\,\cdot\,}_\rk$ denotes the normalized absolute value on
$\rk$, and $\mathrm{val}:\rk^\times \to \Z$ denotes the normalized
valuation on $\rk$.

\subsection{Definable measures}

Let us review some basic measure theory.  Let $X$ be a measurable
space, that is, it is a set with a $\sigma$-algebra $\Sigma$ on it,
namely, $\Sigma$ is a non-empty set of subsets of $X$ which is
closed under taking countable union and taking complement. An
element of $\Sigma$ is called a measurable subset of $X$. A
non-negative measure on $X$ is defined to be a map $\nu: \Sigma\to
[0,\infty]$ which is countably additive, namely,
\[
 \nu\left(\bigsqcup_{i=1}^\infty S_i\right)=\sum_{i=1}^\infty \nu(S_i)\quad \textrm{for all pairwise disjoint elements $S_1, S_2, S_3, \cdots$ of $\Sigma$.}
\]
A complex function $\phi$ on $X$ is said to be measurable if for
each open subset $U$ of $\C$, $\phi^{-1}(U)\in \Sigma$. Write
$\oM(X)$ for the space of all measurable functions on
  $X$. We say that two elements of $\oM(X)$ are equal to each other
  almost everywhere with respect to $\nu$ if they are equal outside
  a set $S\in \Sigma$ with $\nu(S)=0$.
\begin{dfn}\label{defmeas}
A measure $\mu$ on $X$ is a pair $(\nu, f)$, where $\nu$ is a
non-negative measure on $X$, and $f$ is an element of
\[
  \frac{\{\phi\in \oM(X)\mid \abs{\phi} \textrm{ equals $1$ almost everywhere with respect to }\nu\}}{ \{\phi\in \oM(X)\mid \phi \textrm{ equals $0$ almost everywhere with respect to
  }\nu\}}.
\]
Here the denominator is a vector space, and the numerator is a subset of $\oM(X)$  which is stable under translations by the denominator. Therefore the above quotient makes sense.
\end{dfn}

The non-negative measure $\nu$ of Definition \ref{defmeas} is called
the total variation of $\mu=(\nu,f)$, and is denoted by $\abs{\mu}$.
For each measurable function $\phi$ on  $X$, we say that the
integral $\int_X \phi \mu$ converges if $\int_X \abs{\phi}
\nu<\infty$. In this case, the integration
\[
  \int_X \phi\,\mu:=\int_X \phi f\,\nu
  \]
  is a well-defined complex number. For each $Y\in \Sigma$, $Y$ is a measurable space with the
$\sigma$-algebra $\Sigma_Y:=\{S\in \Sigma\mid S\subset Y\}$.  We
define the restriction $\mu|_Y$ of $\mu$ to $Y$ in the obvious way.
For each measurable function $\phi$ on $X$, the multiplication
$\phi\mu$ is defined to be the measure $(\abs{\phi}\nu,
\frac{\phi}{\abs{\phi}}f)$ on $X$.

Note that every semialgebraic space is naturally a measurable space: the
$\sigma$-algebra is generated by all the semialgebraic subsets.

\begin{dfn}
\label{def-measure}
Let $X$ be a semialgebraic space. A measure $\mu$ on $X$ is said to
be definable of order $\leq k$ ($k\in \BN$) if there is a family
$\{f_i: S_i\to X_i\}_{i=1,2,\cdots, r}$ ($r\in \BN$) of isomorphisms
of semialgebraic spaces such that
\begin{itemize}
  \item $S_i$ is a semialgebraic subset of $\rk^{n_i}$, for some $n_i\in \BN$ ($i=1,2,\cdots, r$);
  \item $\{X_i\}_{i=1,2,\cdots, r}$ is a cover of $X$ by its semialgebraic subsets;
  \item for each $i=1,2,\cdots, r$, the restriction of $\mu$ to $S_i$ via $f_i$ has the form $\phi_i \mu_{S_i}$, where $\mu_{S_i}$ denotes the restriction to $S_i$ of a Haar measure of $\rk^{n_i}$,
   and $\phi_i$ is a definable function on $S_i$ of order $\leq k$.
  \end{itemize}
\end{dfn}

Write $\rr$ for the ring of integers in $\rk$. Fix a uniformizer
$\varpi\in \rr$. For each integer $k\geq 1$, put
\[
  \rr_k:=\bigsqcup_{r=0}^\infty \varpi^r (1+\varpi^k \rr).
\]
Then $\mathrm{R}_k$ is a semialgebraic set, since ${\rm R}_k=\{ x\in {\rm F} | x\not= 0,  \text{ and } {\rm ac}(x)\equiv 1 \mod \varpi^k  \}$ (see \cite[Lemma
2.1 (4)]{De}), where ${\rm ac}(x)$ is the annular component of $x$, i.e. ${\rm ac}(x) = x\varpi^{-{\rm ord}(x)}$.
 We say that a semialgebraic function $f$ on $(\rr_k)^n$
($n\in \BN$) is order monomial if there are integers $d_1,
d_2,\cdots, d_n$ and an element $\beta\in \rk$ such that
\[
   \abs{f(x_1,x_2,\cdots, x_n)}_\rk=\abs{\beta x_1^{d_1} x_2^{d_2} \cdots x_n^{d_n}}_\rk\qquad \textrm{for all }\,(x_1,x_2,\cdots, x_n)\in
   (\rr_k)^n.
\]

The following theorem of rectilinearization with good Jacobians is
proved in \cite[Theorem 7]{ClL}.

\begin{prpd}\label{celld} Let $X$ be a semi-algebraic set in $\rk^n$ ($n\in \BN$),  and let $\{f_j\}_{j = 1,\cdots, r}$ ($r\in \BN$) be a family of semialgebraic functions on $X$. Then there exists  a family
\[
\{\phi_i: (\rr_{k_i})^{n_i}\to \rk^n\}_{i=1,2,\cdots, t} \quad(t\in
\BN,\, k_i\geq 1,\,n_i\in \BN)
\]
 of injective semialgebraic maps such that
\begin{itemize}
  \item $\{\phi_i((\rr_{k_i})^{n_i})\}_{i=1,2,\cdots, t}$ forms a partition of $X$;
  \item the restriction of $f_j$ to $(\rr_{k_i})^{n_i}$ through $\phi_i$ is order monomial ($j=1,2,\cdots, r$, $i=1,2,\cdots, t$);  and
  \item for each $i=1,2,\cdots, t$, if $n_i=n$, then  $\phi_i$ is continuously differentiable and their Jacobian is order monomial.
  \end{itemize}
\end{prpd}

We say that a measure $\mu$ on  $(\rr_k)^n$ is simple of order $\leq k$ if
it is a linear combination of measures of the form
\[
  P(\mathrm{val}(x_1), \mathrm{val}(x_2), \cdots, \mathrm{val}(x_n))
  u_1^{\mathrm{val}(x_1)} u_2^{\mathrm{val}(x_2)}\cdots
  u_n^{\mathrm{val}(x_n)}   \mu_{(\rr_k)^n},
\]
where $P$ is a (complex) polynomial of degree $\leq k$, $u_1,
u_2,\cdots, u_n\in \C^\times$, and $\mu_{(\rr_k)^n}$ is the
restriction of a Haar measure on $\rk^n$ to $(\rr_k)^n$.  Clearly if a measure $\mu$ on $(\rr_k)^n$ is simple of order $\leq k$, then $\mu$ is definable of order $\leq k$.

\begin{lemd}\label{zerom}
Every semialgebraic set of dimension $<n$ in $\rk^n$ has measure $0$
with respect to a Haar measure on $\rk^n$.
\end{lemd}
\begin{proof}
Let $S\subset \rk^n$ be a non-empty semialgebraic set. Then there is
a  semialgebraic open subset  $S^\circ$ of $S$ such that
$S^\circ$ is a locally closed and locally analytic submanifold of
$\rk^n$, and $\dim(S\setminus S^\circ)<\dim S$ (see \cite{DV} and
\cite[Section 1.2]{CCL}). Note that the measure of $S^\circ$ is $0$.
Therefore the lemma follows by induction on $\dim S$.

\end{proof}

Proposition \ref{celld} and Lemma \ref{zerom} easily imply the
following proposition.

\begin{prpd}\label{dm}
Let $X$ be a semialgebraic space. Let $\{\mu_i\}_{i=1,2,\cdots, r}$
($r\in \BN$) be a family of definable measures on $X$. Let
$\{f_j\}_{j=1,2,\cdots, s}$ ($s\in \BN$) be a family of
semialgebraic functions on $X$. Assume that $\mu_i$ has order $\leq
d_i$ ($i=1,2,\cdots, r$, $d_i\in \BN$). Then there is a family
$\{\phi_k: (\rr_{m_k})^{n_k}\to X_k\}_{k=1,2,\cdots, t}$ ($t\in \BN$,
$m_k\geq 1$, $n_k\in \BN$) of isomorphisms of semialgebraic spaces
such that
\begin{itemize}
   \item $\{X_k\}_{k=1,2,\cdots, t}$ is a partition of $X$ by its semialgebraic subsets;
  \item  the restriction of $\mu_i$ to $(\rr_{m_k})^{n_k}$ via $\phi_k$ is simple of order $\leq d_i$, and the restriction of $f_j$ to $(\rr_{m_k})^{n_k}$ via $\phi_k$ is order monomial,
  for all $k=1,2,\cdots, t$; $i=1,2,\cdots, r$; $j=1,2,\cdots, s$.
  \end{itemize}
\end{prpd}

\subsection{Igusa zeta integrals}

Write $q_\rk$ for the cardinality of the residue field $\rr/\varpi
\rr$. In this subsection, we prove the following general form of the
convergence and rationality of Igusa zeta integrals.

\begin{thmd}\label{igusa}
Let $\mu$ be a definable measure of order $\leq k$ ($k\in \BN$) on a
semialgebraic space $X$.  Let $f$ be a nowhere vanishing bounded
semialgebraic function on $X$ such that
\[
  \abs{\mu}(X_{f,\epsilon})<\infty\quad \textrm{for all $\epsilon>0$},
 \]
 where $X_{f,\epsilon}:=\{x\in X\mid \abs{f(x)}_\rk> \epsilon\}$.
Then the integral
\[
  \oZ_{\mu}(f, s):=\int_{X} \abs{f}_\rk^s\,\mu\qquad (s\in \C)
 \]
  converges when the real part of $s$ is sufficiently large. Moreover, there exists a meromorphic function
 \[
   M(s)=\frac{P(q_\rk^{-s}, q_\rk^{s} )}{(1-a_1 q_\rk^{-s})^{n_1}(1-a_2 q_\rk^{-s})^{n_2} \cdot \cdots \cdot (1-a_r q_\rk^{-s})^{n_r} }
 \]
 on $\C$,  where
 \begin{itemize}
   \item $r\in \BN$, $a_1, a_2, \cdots, a_r$ are pairwise distinct non-zero complex numbers;
      \item $n_1, n_2, \cdots, n_r\in \{1,2, \cdots,\dim X+k\}$;
      \item $P$ is a two variable polynomial with complex coefficients,
     \end{itemize}
     such that  if $\oZ_{\mu}(f,s)$ is absolutely convergent for $s=s_0\in\C$, then $M(s)$ is holomorphic at $s_0$, and $\oZ_{\mu}(f,s_0)=M(s_0)$.

 \end{thmd}

For each $k\in \BN$, write $\aut_k(\Z^n)$ ($n\in \BN$) for the space
of all complex functions which are linear combinations of the
functions of the form
 \[
    x\mapsto \chi(x) P(x),
 \]
 where $\chi$ is a character of $\Z^n$, and $P$ is a polynomial of degree $\leq k$.

By Proposition \ref{dm}, in order to prove Theorem \ref{igusa}, we
assume without loss of generality that $X=(\rr_{m})^n$ ($m\geq 1$,
 $n\in \BN$), $\mu$ is simple of order $\leq k$, and $f$ is order
 monomial. Then $\mu$ is the multiple of $\mu_X$ with a function
\[
  (x_1,x_2, \cdots, x_n)\mapsto \phi(\mathrm{val}(x_1), \mathrm{val}(x_2),\cdots, \mathrm{val}(x_n)),
\]
where $\mu_X$ denotes the restriction of the normalized Haar measure $\mu$
on $\rr^n$ (i.e. $\mu(\rr^n)=1$) to $X$, and $\phi\in \aut_k(\Z^n)$. Since $f$ is
non-zero, bounded, and order monomial, we have that
\[
     \abs{f(x_1,x_2,\cdots, x_n)}_\rk=q_\rk^{c} \abs{x_1^{d_1} x_2^{d_2} \cdots x_n^{d_n}}_\rk\qquad \textrm{for all }\,(x_1,x_2,\cdots, x_n)\in (\rr_m)^n,
\]
for some  $c\in \Z$ and $d_1, d_2,\cdots, d_n\in \BN$. Then
\begin{eqnarray*}
 && \oZ_{\mu}(f, s) \\
 &=& q_\rk^{c s-m n} \sum_{x=(x_1,x_2, \cdots, x_n)\in \BN^n} q_\rk^{-x_1} q_\rk^{-x_2}\cdot \cdots
  \cdot q_\rk^{-x_n}\phi(x) q_\rk^{-s d_1 x_1} q_\rk^{-s d_2 x_2}\cdot \cdots
  \cdot q_\rk^{-s d_n x_n}.
\end{eqnarray*}
Therefore Theorem \ref{igusa} is implied by the following
Proposition, which will be proved in the next subsection.

 \begin{prpd}\label{igusaz}
 Let $\chi$ be a character on $\Z^n$ of the form
 \[
    (x_1,x_2, \cdots, x_n)\mapsto q_\rk^{-d_{1}x_{1}} q_\rk^{-d_{2}x_{2}}\cdot \cdots \cdot q_\rk^{-d_n x_n},
 \]
 where $d_1, d_2, \cdots, d_n\in \BN$. Let
$\phi\in  \aut_k(\Z^n)$ ($k\in \BN$).  Assume that \be\label{sumep0}
   \sum_{x\in \BN^n, \abs{\chi(x)}>\epsilon} \abs{\phi(x)}<\infty\quad
   \textrm{for all }\epsilon>0.
 \ee
Then the summation
 \[
  \oZ_\phi(\chi,s):=\sum_{x\in \BN^n} \phi(x) \chi(x)^s \qquad (s\in \C)
 \]
 absolutely converges when the real part of $s$ is sufficiently
large. Moreover, there exists a meromorphic function
 \[
   M(s)=\frac{P(q_\rk^{-s})}{(1-a_1 q_\rk^{-s})^{n_1}(1-a_2 q_\rk^{-s})^{n_2} \cdot \cdots \cdot (1-a_r q_\rk^{-s})^{n_r} }
 \]
 on $\C$, where
 \begin{itemize}
   \item $r\in \BN$, $a_1, a_2, \cdots, a_r$ are pairwise distinct non-zero complex numbers;
      \item $n_1, n_2, \cdots, n_r\in \{0,1,2, \cdots,n+k\}$;
      \item $P$ is a polynomial with complex coefficients,
      \end{itemize}
 such that  if $ \oZ_\phi(\chi,s)$ is absolutely convergent for $s=s_0\in\C$, then $M(s)$ is holomorphic at $s_0$, and $ \oZ_\phi(\chi,s_0)=M(s_0)$.
 \end{prpd}
\subsection{Proof of Proposition \ref{igusaz}}

 Write $\aut(\Z^n):=\bigcup_{k=0}^\infty \aut_{k}(\Z^n)$. We view the space $\con(\Z^n)$ of $\C$-valued functions on $\Z^n$ as a
 representation of $\Z^n$ under translations:
 \[
  (x_0.\phi)(x):=\phi(x+x_0), \quad x,x_0\in \Z^n, \,\phi \in \con(\Z^n).
 \]

 \begin{lemd}\label{gpo}
 The space of $\Z^n$-finite vectors in $\con(\Z^n)$ equals to $\aut(\Z^n)$.
\end{lemd}
\begin{proof}
Let $\con(\Z^n)^f$ be the space of all $\Z^n$-finite vectors.  With respect to the action of $\Z^n$, we can decompose $\con(\Z^n)^f$ into the direct sum of generalized eigenspaces,
$$\con(\Z^n)^f =\bigoplus_{\chi}  \con(\Z^n)^\chi.$$
Here $\chi$ is taken over all characters of $\Z^n$, and $\con(\Z^n)^\chi$ consists of functions $\phi$ such that for some $N>0$,
 \[
 (x_1-\chi(x_1))(x_2-\chi(x_2)) \,\cdots \, (x_N-\chi(x_N)). \phi=0\quad \textrm{for all }x_1, x_2, \cdots, x_N\in \Z^n.
 \]
  Then the space $\chi^{-1}\con(\Z^n)^\chi$ exactly consists of the generalized invariant functions on $\Z^n$. It is well-known that the space of generalized invariant functions on $\Z^n$ with respect to the translations action coincide with the space of polynomials on $\Z^n$. This finishes the proof of the lemma.
\end{proof}

 Define $\aut^\circ(\Z^n)$ to be the subspace of $\aut(\Z^n)$ spanned by functions of the form
 \[
    (x_1,x_2, \cdots, x_n)\mapsto u_1^{x_1} u_2^{x_2}\cdot \cdots \cdot u_n^{x_n}P(x_1,x_2,\cdots, x_n),
 \]
 where $u_1, u_2, \cdots, u_n$ are non-zero complex numbers of absolute value $<1$, and $P$ is a polynomial.

 \begin{lemd}\label{a0}
Let $\chi$ be a character on $\Z^n$ of the form
 \[
    (x_1,x_2, \cdots, x_n)\mapsto u_1^{x_1} u_2^{x_2}\cdot \cdots \cdot u_n^{x_n},
 \]
 where $u_1, u_2, \cdots, u_t$ are complex numbers of absolute value $1$, and $u_{t+1}, u_{t+2}, \cdots, u_n$ are complex numbers of absolute value $<1$ ($0\leq t\leq n$). Let
$\phi\in  \aut(\Z^n)$. Then
 \be\label{sumep}
   \sum_{x\in \BN^n, \abs{\chi(x)}>\epsilon} \abs{\phi(x)}<\infty\quad
   \textrm{for all }\epsilon>0
 \ee
 if and only if $\phi\in \aut^\circ(\Z^t)\otimes \aut(\Z^{n-t})$.
 \end{lemd}
\begin{proof}
It is easy to see that \eqref{sumep} holds if and only if
\be\label{sumep1}
   \sum_{x\in \BN^{t}} \abs{\phi(x,x')}<\infty\quad
   \textrm{for all $x'\in \BN^{n-t}$}.
 \ee
 Write $A$ for the space of all $\phi\in \aut(\Z^n)$ such that for some
$x_0\in \Z^t$ and $x_0'\in \Z^{n-t}$,
 \be\label{sumep2}
   \sum_{x\in x_0+\BN^{t}} \abs{\phi(x,x')}<\infty\quad
   \textrm{for all $x'\in x_0'+\BN^{n-t}$}.
 \ee
The space $A$ is a $\Z^n$-subrepresentation of $\aut(\Z^n)$ containing $
\aut^\circ(\Z^t)\otimes \aut(\Z^{n-t})$.

Note that every one
dimensional $\Z^n$-subrepresentation of $A$ is contained in
$\aut^\circ(\Z^t)\otimes \aut(\Z^{n-t})$. We also note that  $\aut^\circ(\Z^t)\otimes \aut(\Z^{n-t})$ is  closed under the multiplication by polynomials on $\Z^n$.  It implies that
$A\subset \aut^\circ(\Z^t)\otimes \aut(\Z^{n-t})$, by considering the generalized eigenspace decomposition of $\aut(\Z^n)$ under the action of $\Z^n$.  Hence the space $A$ is exactly identical to $\aut^\circ(\Z^t)\otimes \aut(\Z^{n-t})$.

On the other
hand, it is obvious that for all $\phi\in \aut^\circ(\Z^t)\otimes
\aut(\Z^{n-t})$, \eqref{sumep2} holds for all $x_0\in \Z^t$ and
$x_0'\in \Z^{n-t}$,  in particular  $\eqref{sumep1}$ holds.
 This proves the lemma.

\end{proof}

The following lemma is easy to check and we omit the details.
 \begin{lemd}\label{elemsum}
  Let $u$ be a non-zero complex number of absolute value $<1$. Then
 \[
  \sum_{x\in \BN} \left(\!\begin{array}{c}\!x \!\\ \! k\!\end{array}\!\right) u^x=\frac{u^k}{(1-u)^{k+1}}
  \quad \textrm{for all }k\in \BN.
 \]
 \end{lemd}

Now we come to the proof of Proposition \ref{igusaz}. Without loss
of generality, assume that $d_1,d_2, \cdots d_t$ are all $0$, and
$d_{t+1}, d_{t+2}, \cdots, d_n$ are all positive ($0\leq t\leq n$).
By Lemma \ref{a0}, the assumption $\eqref{sumep0}$ implies that $\phi\in \aut^\circ(\Z^t)\otimes
\aut(\Z^{n-t})$. Lemma \ref{a0}  also implies that
$\oZ_\phi(\chi,s)$ is absolutely convergent if and only if
 \be
\label{phiconv}
    \phi\chi^s \in \aut^\circ(\Z^n)= \aut^\circ(\Z^t)\otimes \aut^\circ(\Z^{n-t}),
\ee
 where $\chi^s$ denotes the character
\[
  \Z^{n}\to \C^\times,\quad  (x_{1},x_{2}, \cdots, x_n)\mapsto q_\rk^{-s d_{t+1}x_{t+1}} q_\rk^{-s d_{t+2}x_{t+2}}\cdot \cdots \cdot q_\rk^{-s d_n x_n}.
 \]
 It is clear that \eqref{phiconv} holds when the real part of $s$ is
 large. This proves the first assertion of the proposition. Now assume
 that \eqref{phiconv} holds. Without loss of generality, we further assume that
 \[
 \phi(x_1, x_2, \cdots, x_n)= \left(\!\begin{array}{c}\!x_1 \!\\ \!
 k_1\!\end{array}\!\right)  \left(\!\begin{array}{c}\!x_2 \!\\ \!
 k_2\!\end{array}\!\right)\cdot \cdots \cdot  \left(\!\begin{array}{c}\!x_n \!\\ \!
 k_n\!\end{array}\!\right) q_\rk^{\alpha_1 x_1}q_\rk^{\alpha_2 x_2} \cdot \cdots \cdot q_\rk^{\alpha_n
 x_n},
 \]
 for all $x_1,x_2, \cdots, x_n\in \Z^n$, where $\alpha_i\in \C$, $k_i\in \BN$ ($i=1,2,\cdots, n$) and $k_1+k_2+\cdots +k_n\leq
 k$. Using Lemma \ref{elemsum}, it is now easy to see that the
 summation $\oZ_\phi(\chi, s)$ has the desired property.

\subsection{Semialgebraic $\ell$-spaces and meromorphic
continuations of distributions}

 \begin{dfn}
 A semialgebraic $\ell$-space (over $\rk$) is a Hausdorff topological space $X$ which is at the same time a semialgebraic space (over $\rk$), with the following property: there is a  finite family of semialgebraic charts $\{(U_i,U'_i,\phi_i)\}_{i=1,2,\cdots, r}$ ($r\in \BN$) of $X$ such that
 \begin{itemize}
   \item for all $i=1,2,\cdots, r$, $U_i$ is  locally closed in $\rk^{n_i}$ for some $n_i\in \BN$, $U_i'$ is open in $X$, and $\phi_i$ is a homeomorphism; and
   \item  $X=U'_1\cup U'_2\cup \cdots \cup U_r'$.
 \end{itemize}
\end{dfn}

 It is clear that every semialgebraic $\ell$-space is an $\ell$-space; the product of two  semialgebraic $\ell$-spaces is still a  semialgebraic $\ell$-space;
 and a locally closed semialgebraic subset of a  semialgebraic $\ell$-space is a  semialgebraic $\ell$-space.
 All semialgebraic $\ell$-spaces form a category whose morphisms are semialgebraic continuous
 maps. For each algebraic variety $\mathsf X$ over $\rk$,
 $\mathsf X(\rk)$ is obviously a  semialgebraic $\ell$-space.
 Note that every semialgebraic $\ell$-space is naturally a
 measurable space, with
 the $\sigma$-algebra  generated by all the
open sets, which coincides with the one generated by all semialgebraic sets. An element of this $\sigma$-algebra is called a Borel
subset of the semialgebraic $\ell$-space.

Recall the following Riesz representation theorem.

\begin{thmd}\label{riesz}

Let $X$ be a locally compact Hausdorff topological space which is
locally secondly countable, namely, every point of $X$ has a
neighborhood which is secondly countable as a topological space.
Then the map
\[
  \begin{array}{rcl}
   \{\textrm{locally finite measures on $X$}\}&\to &
   \{\textrm{continuous linear functionals on $\con_c(X)$}\},\\
    \mu& \mapsto& (\phi \mapsto \int_X \phi \mu)
  \end{array}
\]
is bijective.
\end{thmd}
Here $\con_c(X)$ denotes the space of all compactly supported
continuous functions on $X$, with the usual inductive topology. A
measure $\mu$ on $X$ is said to be locally finite if
\[
        \abs{\mu}(K)<\infty \quad \textrm{for every compact subset $K$ of $X$.}
\]
Recall that  every locally finite measure on $X$ is regular, as $X$ is assumed to be locally secondly countable (see \cite[Proposition 7.2.3]{Co}).

If $X$ is an $\ell$-space, then $\oS(X)$ is a
dense subspace of $\con_c(X)$. Write $\oD(X):=
\Hom(\oS(X), \C)$ for the space of distributions on $X$.
By Theorem \ref{riesz}, we have an embedding
\[
   \{\textrm{locally finite measure on $X$}\}\hookrightarrow
   \oD(X).
\]
Using this embedding, we view every locally finite measure on $X$ as
a distribution on it.

 Now let $X$ be a semialgebraic
$\ell$-space and let $f$ be a continuous semialgebraic function on
$X$. Write $X_f:=\{x\in X\mid f(x)\neq 0\}$, which is also a
semialgebraic $\ell$-space. Let $\mu$ be a  locally finite definable
measure on $X_f$ of order $\leq k$ ($k\in \BN$). Let $\phi\in
\oS(X)$. Note that $\phi$ is definable of order $\leq 0$.
Theorem \ref{igusa} implies that the integral
\[
  \oZ_{\mu,f}(\phi,s):=\int_{X_f}\abs{f}_\rk^s \phi \mu
\]
defines a meromorphic function on $\C$. Moreover, for each $a_0\in \C^\times$, $\oZ_{\mu,f}(\phi,s)$ is a rational function of $1-a_0 q_\rk^{-s}$. Therefore we have the Laurent expansion
\[
  \oZ_{\mu,f}(\phi,s)=\sum_{i\in \Z} \oZ_{\mu,f,a_0,i}(\phi)
  (1-a_0 q_\rk^{-s})^i.
\]
We are interested in
the distribution $\oZ_{\mu,f,a_0,i}$ on $X$.   Note that $\oZ_{\mu,f,a_0,i}=0$ when $i<-(\dim X+ k)$.

\subsection{The invariance property of $\oZ_{\mu,f,a_0,i}$}
\label{Invariance_property}
Let $G$ be an abstract group which acts as automorphisms of the
semialgebraic $\ell$-space $X$. For every $g\in G$, and every distribution $\eta$ on $X$ (or on some $G$-stable locally closed subset of $X$), write $g*\eta$ for the push forward of $\eta$ through the action of $g$. It is clear that for every distribution $\eta$ on $X$, $\eta\in \Hom_{G, k}(\oS(X), \chi)$ if and only if
$$\left((g_0-\chi(g_0^{-1}))(g_1-\chi(g_1^{-1}))\cdots  (g_k-\chi(g_k^{-1}))\right)*\eta=0,$$
 for all  $ g_0, g_1,\cdots, g_k\in G$.

Now assume that there is a locally constant
homomorphism  $\chi_f:G\rightarrow \Z$ such that \be\label{invfa}
   \abs{f(g.x)}_\rk=q_\rk^{\chi_f(g)}\abs{f(x)}_\rk,\qquad g\in G, \,x\in X,
\ee and there is a   character $\chi_\mu$ on $G$ such that
\[
   \mu\in \Hom_{G,k'}(\oS(X_f), \chi_\mu), \quad \textrm{for some  $k'\in
   \BN$.}
\]

\begin{prpd}
\label{Zeta_invariant}
Let $a_0\in \C$. Let $i_0$ be an integer so that
$\oZ_{\mu,f,a_0,i}=0$ for all $i<i_0$. Then
\[
  \oZ_{\mu,f,a_0,i}\in \Hom_{G,k'+i-i_0}(\oS(X), \chi_\mu a_0^{\chi_f} )\quad \textrm{for all  $i\geq i_0$.}
\]
\end{prpd}
\begin{proof}
For each locally finite definable measure $\mu'$ on $X_f$, write
$\oZ_{\mu'}(s)$ for the following distribution on $X$:
\[
  \phi\mapsto  \oZ_{\mu',f}(\phi,s):=\int_{X_f}\abs{f}_\rk^s \phi
  \mu'.
\]
For each $i\in \Z$, write $\oZ_{\mu',i}$ for the $i$-th coefficients
of the  Laurent expansion of $\oZ_{\mu'}(s)$ as a rational function of $1-a_0q_\rk^{-s}$. It is a
distribution on $X$.

 The invariance property of $\abs{f}_\rk$
implies that
\[
 (g-\chi_\mu(g^{-1})q_\rk^{-\chi_f(g)s})* \oZ_{\mu}(s)=q_\rk^{-\chi_f(g)s}\oZ_{(g-\chi_{\mu}(g^{-1}))*\mu} ,\quad \textrm{for all $g\in G$}.
\]
 Comparing the Laurent expansions
of the two sides of the above equality, we know that
\[
   (g-\chi_\mu(g^{-1})a_0^{-\chi_f(g)}))*\oZ_{\mu,i} -a_0^{-\chi_f(g)}\oZ_{(g-\chi_\mu(g^{-1}))*\mu,i}
\]
is a linear combination of distributions of the form
$g'*\oZ_{\mu,i'}$, where $i'<i$ and $g'\in G$. Then the proposition
follows by induction on $k'+i-i_0$.
\end{proof}

\begin{cord}\label{corextm}
 Let $\chi$ be a character of $G$. Then every generalized $\chi$-invariant locally finite definable  measure  on $X_f$
  extends to a generalized $\chi$-invariant distribution on $X$.

\end{cord}
\begin{proof}
Write $\mu$ for the measure of the proposition. The distribution
$\oZ_{\mu,f,1,0}$ on $X$ extends $\mu$ and is generalized
$\chi$-invariant.
\end{proof}

\section{Generalized invariant functions and definable measures}
\label{sec_main2}
\subsection{Generalized functions on homogeneous spaces}
Let $G$ be an $\ell$-group and let $X$ be a homogeneous space of it.
We say that a distribution $\eta$ on $X$ is smooth if for every
$x\in X$, there is an open compact subgroup $K$ of $G$ such that
$\eta|_{K.x}$ is $K$-invariant. Denote by $\oD_c^\infty(X)$ the
space of all smooth distributions on $X$ with compact support. A
generalized function on $X$ is defined to be a linear functional on
$\oD_c^\infty(X)$. The space of all generalized functions on $X$ is
denoted by $\con^{-\infty}(X)$. As before, the space of all
distributions on $X$ is denoted by $\oD(X)$.

The following lemma is elementary and we omit its proof.
\begin{lem}\label{compfd}
Let $\eta$ be a smooth distribution on $X$ which has non-zero
restriction to all non-empty open subset of $X$. Then the map
\[
\begin{array}{rcl}
 \con^{-\infty}(X)&\to& \oD(X),\\
 f&\mapsto & f\eta:=(\phi\mapsto f(\phi \eta))
\end{array}
\]
is a linear isomorphism.
\end{lem}

Using the following injective linear map, we view every locally constant
function on $X$ as a generalized function on $X$:
\[
  \begin{array}{rcl}
 \con^\infty(X)&\to& \con^{-\infty}(X),\\
 f&\mapsto & (\eta \mapsto \eta(1_\eta f)),
\end{array}
\]
where $\con^\infty(X)$ is the space of locally constant functions on $X$, and  $1_\eta$ denotes the characteristic function of the support of
$\eta$.

\begin{lem}\label{smoothg}
Let $K$ be an open compact subgroup of $G$. Then every $K$-invariant
generalized function on $X$ is a locally constant function on $X$.
\end{lem}
\begin{proof}
Without loss of generality, assume that $G=K$. Then in view of Lemma
\ref{compfd}, the lemma follows easily by the existence and
uniqueness of $K$-invariant distributions on $X$.
\end{proof}

\subsection{Characters on algebraic homogeneous spaces}
\label{secch} Let $\mathsf G$ be a linear algebraic group over
$\rk$, with an algebraic subgroup $\mathsf H$ of it. Denote by
$\mathsf N$ the unipotent radical of $\mathsf G$. Write
$G:=\mathsf G(\rk)$, $H:=\mathsf H(\rk)$ and $N:=\mathsf
N(\rk)$.

In this subsection, we prove the following proposition.
\begin{prpl}
\label{dch} Assume that $\mathsf G$ is connected.  Let $\chi$ be a
character on $G$ which is trivial on $N$ and has finite order when
restricted to $H$. Then $\chi$ has the form
\[
\abs{\beta_1}_\rk^{s_1} \cdot \abs{\beta_2}_\rk^{s_2}
   \cdot \cdots \cdot
   \abs{\beta_t}_\rk^{s_t} \cdot \chi_\mathrm f,
\]
where $t\in \BN$, $\beta_1, \beta_2, \cdots, \beta_t$ are algebraic
characters on $\mathsf G$ which are trivial on $\mathsf H$,
$s_1,s_2, \cdots, s_t\in \C$, and $\chi_\mathrm f$ is a finite order
character on $G$.
\end{prpl}

The following lemma is obvious.

\begin{lem}\label{ctorus}
 Let $\mathsf A$ be a split algebraic torus over $\rk$.
Then every character on $\mathsf A(\rk)$ has the form
\[
\abs{\beta_1}_\rk^{s_1} \cdot \abs{\beta_2}_\rk^{s_2}
   \cdot \cdots \cdot
   \abs{\beta_t}_\rk^{s_t} \cdot \chi_\mathrm f,
\]
where $t\in \BN$, $\beta_1, \beta_2, \cdots, \beta_t$ are algebraic
characters on $\mathsf A$, $s_1,s_2, \cdots, s_t\in \C$, and
$\chi_\mathrm f$ is a finite order character on $\mathsf A$.
\end{lem}

Generalizing Lemma \ref{ctorus}, we have the following lemma.
\begin{lem}\label{ctorus2}
Let $\mathsf A$ be a split algebraic torus over $\rk$,
with an algebraic subgroup $\mathsf S$ of it.  Let $\chi$ be a
character on $\mathsf A(\rk)$ which has finite order when
restricted to $\mathsf S(\rk)$. Then $\chi$ has the form
\[
\abs{\beta_1}_\rk^{s_1} \cdot \abs{\beta_2}_\rk^{s_2}
   \cdot \cdots \cdot
   \abs{\beta_t}_\rk^{s_t} \cdot \chi_\mathrm f,
\]
where $t\in \BN$, $\beta_1, \beta_2, \cdots, \beta_t$ are algebraic
characters on $\mathsf A/\mathsf S$, $s_1,s_2, \cdots, s_t\in
\C$, and $\chi_\mathrm f$ is a finite order character on $\mathsf
A(\rk)$.
\end{lem}

\begin{proof}
Denote by $\mathsf S_0$ the identity connected component of
$\mathsf S$, which is also a split algebraic torus. Then there is
an algebraic subtorus $\mathsf S'$ of $\mathsf A$ such that
$\mathsf A=\mathsf S_0\times_\rk \mathsf S'$. By Lemma
\ref{ctorus}, $\chi|_{\mathsf S'(\rk)}$ has the form
\[ \abs{\beta_1}_\rk^{s_1} \cdot
\abs{\beta_2}_\rk^{s_2}
   \cdot \cdots \cdot
   \abs{\beta_t}_\rk^{s_t} \cdot \chi'_\mathrm f,
\]
where $t\in \BN$, $\beta_1, \beta_2, \cdots, \beta_t$ are algebraic
characters on $\mathsf S'$, $s_1,s_2, \cdots, s_t\in \C$, and
$\chi'_\mathrm f$ is a finite order character on $\mathsf
S'(\rk)$. The group $\mathsf A/\mathsf S$ is obviously
identified with a quotient group of $\mathsf S'$, and there is a
positive integer $m$ such that $\beta_1^m, \beta_2^m, \cdots,
\beta_t^m$ descends to
 algebraic characters on $\mathsf A/\mathsf S$. Then we have
 that
 that
 \[
  \chi=\chi|_{\mathsf S_0(\rk)}\otimes
   \left(\abs{\beta_1^m}_\rk^{s_1/m} \cdot \abs{\beta_2^m}_\rk^{s_2/m}
   \cdot \cdots \cdot
   \abs{\beta_t^m}_\rk^{s_t/m} \cdot \chi'_\mathrm f\right).
 \]
This proves the lemma.
\end{proof}

\begin{lem}\label{fs}
For each surjective algebraic homomorphism $\mathsf G\to \mathsf
G'$ of linear algebraic groups over $\rk$, the image of the induced
group homomorphism $\mathsf G(\rk)\to \mathsf G'(\rk)$ has
finite index in $\mathsf G'(\rk)$.
\end{lem}
\begin{proof}
This is a direct consequence of Lemma \ref{finitehom}.
\end{proof}

\begin{lem}\label{raf}
Assume that $\mathsf G$ is connected. Let $\mathsf A$ be the
largest central split torus in a Levi component $\mathsf L$ of
$\mathsf G$. Let $\chi$ be a character on $G$ which is trivial on
$N$. Then $\chi$ has finite order if and only if its restriction to
$\mathsf A(\rk)$ has finite order.
\end{lem}
\begin{proof}
The ``only if" part of the lemma is trivial. We prove the ``if"
part. Assume that $\chi|_{\mathsf A(\rk)}$ has finite order.  Let
$\mathsf S$ denote the simply connected covering of the derived
subgroup of $\mathsf L$, let $\mathsf T$ denote the maximal
anisotropic central torus in $\mathsf L$. Then by Lemma \ref{fs},
the image of the multiplication map \be \label{decomg2}
 \varphi: (\mathsf S(\rk)\times \mathsf T(\rk)\times \mathsf A(\rk))\ltimes N\to G=\mathsf G(\rk)
\ee has finite index in  $G$. Therefore it suffices to show that
$\chi\circ \varphi$ has finite order. This holds because $\mathsf
S(\rk)$ is a perfect group, $\mathsf T(\rk)$ is compact,
$\chi|_{\mathsf A(\rk)}$ has finite order, and $\chi|_N$ is
trivial.

\end{proof}

We are now ready to prove Proposition \ref{dch}.
\begin{proof}
 Assume that
$\mathsf G$ is connected  and let $\mathsf A$ be as in Lemma
\ref{raf}. Let $\chi$ be as in Proposition \ref{dch}. Write
$\mathsf G'$ for the largest quotient of $\mathsf G$ which is a
split algebraic torus. Consider the commutative diagram
\[
  \begin{CD}
           \mathsf A @>\varphi>>\mathsf G@>\varphi' >>\mathsf G'\\
            @AAA           @AAA @AAA\\
          \mathsf S:=(\varphi'\circ\varphi)^{-1}(\mathsf H') @>>> \varphi'^{-1}(\mathsf H')=\mathsf H \cdot \ker \varphi' @>>>\mathsf H':=\varphi'(\mathsf H),
  \end{CD}
\]
where $\varphi$ denotes the inclusion homomorphism, $\varphi'$
denotes the quotient homomorphism, and the vertical arrows are
inclusion homomorphisms. As in the proof of Lemma \ref{raf}, we know
that $\chi$ has finite order when restricted to
$(\varphi'^{-1}(\mathsf H'))(\rk)$. In particular,
$\chi|_{\mathsf S(\rk)}$ has finite order. By Lemma \ref{ctorus2},
there are algebraic characters $\beta_1, \beta_2, \cdots, \beta_t$
($t\in \BN$) on $\mathsf A/\mathsf S$ and $s_1,s_2, \cdots,
s_t\in \C$ so that the character
\[
  \chi|_{\mathsf S(\rk)}\cdot \left(\abs{\beta_1}_\rk^{s_1} \cdot
\abs{\beta_2}_\rk^{s_2}
   \cdot \cdots \cdot
   \abs{\beta_t}_\rk^{s_t}\right)^{-1},
\]
on $\mathsf S(\rk)$ has finite order. Since $\mathsf A/\mathsf
S=\mathsf G'/\mathsf H'=\mathsf G/(\mathsf H\cdot\ker \varphi')$, we may view
$\beta_1, \beta_2, \cdots, \beta_t$ as algebraic characters on
$\mathsf G$ which are trivial on $\mathsf H$. Therefore
Proposition \ref{dch} follows by Lemma \ref{raf}.
\end{proof}

\begin{prpl}
\label{dch2} Assume that $\mathsf G$ is connected.  Let $\chi':
G\to \C$ be a locally constant group homomorphism which is trivial
on $H$. Then $\chi'$ is a linear combination of the characters of
the form $\mathrm{val}\circ \alpha$, where $\alpha$ is an  algebraic
characters on $\mathsf G$ which are trivial on $\mathsf H$.
\end{prpl}
\begin{proof}
Note that $\C$ has no nontrivial finite subgroup. This implies that
$\chi'|_N$ is trivial. Then the proposition is proved by the same
argument of the proof of Proposition \ref{dch}
\end{proof}

\subsection{Generalized invariant functions on algebraic homogeneous spaces}
\label{secch2} We continue with the notation of the last subsection. Put
$X:=G/H$.

This subsection is to prove the following proposition.
\begin{prpl}
\label{H} Assume that $\mathsf G$ is connected. Let $\chi$ be a
character on $G$ which is trivial on $N$. Then every non-zero element
of $\Hom_{G,k}(\oD_c^\infty(X), \chi)$ ($k\in \BN$) is a smooth
function on $X$ of the form
\[
   P(\mathrm{val}\circ \alpha_1, \mathrm{val}\circ \alpha_2, \cdots, \mathrm{val}\circ \alpha_r) \cdot \abs{\beta_1}_\rk^{s_1} \cdot \abs{\beta_2}_\rk^{s_2}
   \cdot \cdots \cdot
   \abs{\beta_t}_\rk^{s_t} \cdot \chi_\mathrm f,
\]
where $r,t\in \BN$, $\alpha_1, \alpha_2, \cdots, \alpha_r$ and
$\beta_1, \beta_2, \cdots, \beta_t$ are algebraic characters on
$\mathsf G$ which are trivial on $\mathsf H$, $s_1,s_2, \cdots,
s_t\in \C$, $P$ is a polynomial of degree $\leq k$, and
$\chi_\mathrm f$ is a finite order character on $G$ which is trivial
on $H$ such that
\[
  \chi= \abs{\beta_1}_\rk^{s_1} \cdot \abs{\beta_2}_\rk^{s_2}
   \cdot \cdots \cdot
   \abs{\beta_t}_\rk^{s_t} \cdot \chi_\mathrm f.
\]
\end{prpl}

First we have the following lemma.

\begin{lem}
\label{smoothg2} Let $\chi$ be a character on $G$ which is trivial
on an open compact subgroup $K$ of $G$. Then every element $f\in
\Hom_{G,k}(\oD_c^\infty(X), \chi)$ ($k\in \BN$) is a $K$-invariant
smooth function on $X$.
\end{lem}
\begin{proof}
Lemma \ref{homg} implies that $f$ is $K$-invariant. Therefore $f$ is
a smooth function by  Lemma \ref{smoothg}.

\end{proof}

\begin{lem}
\label{smooth} Let $\chi$ be a character on $G$. If the space
$\Hom_{G,\infty}(\oD_c^\infty(X), \chi)$ is non-zero, then $\chi|_H$
is trivial.
\end{lem}
\begin{proof}
Assume that $\Hom_{G,\infty}(\oD_c^\infty(X), \chi)\neq 0$. Then
$\Hom_{G}(\oD_c^\infty(X), \chi)\neq 0$. By Lemma \ref{smoothg2}, we
have a non-zero smooth function $f$ on $X$ such that
\[
  f(g.x)=\chi(g)f(x),\quad \textrm{for all }g\in G, \,x\in X.
\]
This implies that $\chi|_H$ is trivial.

\end{proof}

For each free abelian group $\Lambda$ of finite rank, we view the
space $\con(\Lambda)$ of all complex functions on $\Lambda$ as a
representation of $\Lambda$ by left translations:
 \[
    g.f(x):=f(-g+x),\qquad g,x\in \Lambda, \,f\in \con(\Lambda).
 \]

\begin{lem}\label{poly0}
Let $\Lambda$ be a  free abelian group of finite rank, with a
subgroup $\Lambda_0$ of it. Then every element in
$\con(\Lambda)^{\Lambda, k}\cap
 \con(\Lambda)^{\Lambda_0}$ ($k\in \BN$) has the form
 \[
  P(\lambda_1, \lambda_2, \cdots, \lambda_r)
 \]
 where $r\in \BN$, $\lambda_1, \lambda_2,\cdots, \lambda_r$ are
 group homomorphisms from $\Lambda/\Lambda_0$ to $\C$, and $P$ is a
 polynomial of degree $\leq k$.
\end{lem}
\begin{proof}  Lemma \ref{gpo} easily implies that
\[
  \con(\Lambda)^{\Lambda, k}=\{\textrm{polymonial functions on $\Lambda$ of
  degree $\leq k$}\},\quad (k\in \BN).
 \]
 Then it is elementary to see that the lemma holds.
\end{proof}

\begin{lem}
\label{H2} Assume that $\mathsf G$ is connected. Then every
element $f$ of $\Hom_{G,k}(\oD_c^\infty(X), \C)$ ($k\in \BN$) is a
smooth function on $X$ of the form
\[
   P(\mathrm{val}\circ \alpha_1, \mathrm{val}\circ \alpha_2, \cdots, \mathrm{val}\circ \alpha_r)
\]
where $r\in \BN$, $\alpha_1, \alpha_2, \cdots, \alpha_r$ are
algebraic characters on $\mathsf G$ which are trivial on
$\mathsf H$, and $P$ is a polynomial of degree $\leq k$.
\end{lem}
\begin{proof}
By Lemma \ref{smoothg2}, $f$ is a $G^\circ$-invariant function on
$X$. We identify it with a function on $\Lambda_G$ which is
$[H]$-invariant, where $[H]$ denotes the image of $H$ under the
quotient homomorphism $G\to \Lambda_G$. Then
\[
  f\in \con(\Lambda_G)^{\Lambda_G, k}\cap
 \con(\Lambda_G)^{[H]}.
\]
Therefore the lemma follows by combining Lemma \ref{poly0} and
Proposition \ref{dch2}.
\end{proof}

Now we prove Proposition \ref{H}. Let $f$ be a non-zero element of
$\Hom_{G,k}(\oD_c^\infty(X), \chi)$ ($k\in \BN$). We view $\chi$ as
a function on $X$ since $\chi|_H$ is trivial by Lemma \ref{smooth}.
Then
\[
  \chi^{-1}\cdot f\in \Hom_{G,k}(\oD_c^\infty(X), \C).
\]
Therefore Proposition \ref{H} follows by combining Lemma \ref{H2}
and Proposition \ref{dch}.

\subsection{Nash manifolds and volume forms}
 \begin{dfn}
 A Nash manifold  (over $\rk$) is a locally analytic manifold $X$ over $\rk$ which is at the same time a semialgebraic space (over $\rk$) with the following property: there is a  finite family of semialgebraic charts $\{(U_i,U'_i,\phi_i)\}_{i=1,2,\cdots, r}$ ($r\in \BN$) of $X$ such that
 \begin{itemize}
   \item for all $i=1,2,\cdots, r$, $U_i$ is an  open semialgebraic subset in $\rk^{n_i}$ for some $n_i\geq 0$, $U_i'$ is open in $X$, and $\phi_i$ is a locally analytic diffeomorphism; and
   \item  $X=U'_1\cup U'_2\cup \cdots \cup U_r'$.
 \end{itemize}
\end{dfn}

All Nash manifolds form a category whose morphisms are Nash maps
(namely, locally analytic semialgebraic maps). Every Nash manifold
is clearly a semialgebraic $\ell$-space. Let $X$ be a Nash manifold.
Then the tangent bundle
\[
\oT(X)= \bigsqcup_{x\in X}\oT_x(X)
\]
and the cotangent bundle
\[
\oT^*(X)= \bigsqcup_{x\in X}\oT_x^*(X)
\]
are both naturally Nash manifolds. Therefore \be\label{top0}
   \wedge^\mathrm{top} \oT^*(X):=\bigsqcup_{x\in X} \wedge^{\dim \oT_x(X)} \oT_x^*(X)
\ee is also a Nash manifold. Consequently, for each $m\geq 1$, the
line bundle $(\wedge^\mathrm{top} \oT^*(X))^{\otimes m}$ is also a
Nash manifold.   By a Nash $m$-volume form on $X$, we mean a Nash
section of the bundle $(\wedge^\mathrm{top} \oT^*(X))^{\otimes m}$
over $X$. Fix a Haar measure $\mu_\rk$ on $\rk$. Attach to a Nash
$m$-volume form $\omega$ on $X$, we have a non-negative locally
finite measure $\abs{\omega}_\rk^{\frac{1}{m}}$ on $X$ as usual: in
local coordinate, if
\[
\omega=f(x_1, x_2, \cdots, x_n)(\od x_1\wedge \od x_2\wedge\cdots
\wedge \od x_n)^{\otimes m}, \]
 then
\[
  \abs{\omega}_\rk^{\frac{1}{m}}=\abs{f(x_1, x_2, \cdots, x_n)}_\rk^{\frac{1}{m}} \od\mu_\rk(x_1)\otimes \od \mu_\rk(x_2)\otimes\cdots
\otimes \od \mu_\rk(x_n).
\]
The following lemma is clear.
\begin{lem}\label{nashdf}
For each Nash $m$-volume form $\omega$ on $X$ ($m\geq 1$), the
measure  $\abs{\omega}_\rk^{1/m}$ is locally finite and definable of
order $\leq 0$.
\end{lem}

As usual, write $\CO_\mathsf Y$ for the structure sheaf of an
algebraic variety $\mathsf Y$ over $\rk$. Let $\mathsf X$ be a
smooth algebraic variety over $\rk$. Let $\Omega_{\mathsf
X}:=\Omega_{\mathsf X/\rk}$ denote the sheaf of algebraic
differential forms on $\mathsf X$. Similar to \eqref{top0}, we
define $\wedge^\mathrm{top} \Omega_{\mathsf X}$, which is a
locally free $\CO_\mathsf X$-module of rank $1$.  By an algebraic
$m$-volume form on $\mathsf X$, we mean a global section of the
sheaf $(\wedge^\mathrm{top} \Omega_{\mathsf X})^{\otimes m}$ over
$\mathsf X$.  The notion of  algebraic $1$-volume form exactly coincides with the usual notion of algebraic
volume form.

Given an algebraic $m$-volume form $\omega$ on
$\mathsf X$, a Nash $m$-volume form on the Nash manifold
$\mathsf X(\rk)$ is obviously associated to it. We define
$\abs{\omega}_\rk^{\frac{1}{m}}$ to be the non-negative locally
finite measure on $\mathsf X(\rk)$ attach to this Nash $m$-volume
form.

\subsection{Generalized invariant distributions on algebraic homogeneous spaces}
We continue with the notation of Section \ref{secch} and  Section
\ref{secch2}. Then $X=G/H$ is naturally a Nash manifold since it is
a semialgebraic open subset of $(\mathsf G/\mathsf H)(\rk)$. In
this subsection, we prove the following theorem.

\begin{thml}\label{defm}
Assume that $\mathsf G$ is connected. Let $\chi$ be a character on
$G$ which is trivial on $N$. Then every element of $\Hom_{G,k
}(\oS(X),\chi)$ ($k\in \BN$) is a measure on $X$ and is of the
form
\[
   P(\mathrm{val}\circ \alpha_1, \mathrm{val}\circ \alpha_2, \cdots, \mathrm{val}\circ \alpha_r) \cdot \abs{\beta_1}_\rk^{s_1} \cdot \abs{\beta_2}_\rk^{s_2}
   \cdot \cdots \cdot
   \abs{\beta_t}_\rk^{s_t} \cdot \chi_\mathrm f \cdot (\abs{\omega}_\rk^{1/m})|_X,
\]
where $r,t\in \BN$, $\alpha_1, \alpha_2, \cdots, \alpha_r$ and
$\beta_1, \beta_2, \cdots, \beta_t$ are algebraic characters on
$\mathsf G$ which are trivial on $\mathsf H$, $s_1,s_2, \cdots,
s_t\in \C$, $P$ is a polynomial of degree $\leq k$, $\chi_\mathrm f$
is a finite order character on $G$ which is trivial on $H$, $m\geq
1$, and $\omega$ is an algebraic $m$-volume form on $\mathsf
G/\mathsf H$ which is $\delta$-invariant for some algebraic
character $\delta$ of $\mathsf G$ defined over $\rk$ with the property that
\[
\chi= \abs{\beta_1}_\rk^{s_1} \cdot \abs{\beta_2}_\rk^{s_2}
   \cdot \cdots \cdot
   \abs{\beta_t}_\rk^{s_t} \cdot \chi_\mathrm f\cdot \abs{\delta}_\rk^{\frac{1}{m}}.
\]
\end{thml}

Here the algebraic $m$-volume form $\omega$ on $\mathsf
G/\mathsf H$ is $\delta$-invariant means that
  \[
  g. \omega_{\bar \rk}=\delta(g^{-1}) \omega_{\bar \rk}, \quad \textrm{for all $g\in \mathsf
  G(\bar \rk)$},
\]
where $\omega_{\bar \rk}$ denotes the base extension to $\bar \rk$ of $\omega$.  We also say that an algebraic $m$-volume form $\omega$ is semi-invariant if it is  $\delta$-invariant for some algebraic character $\delta$ of $\mathsf G$ defined over $\rk$.

For the proof of Theorem \ref{defm}, in the rest of this subsection,
we assume that $\mathsf G$ is
connected. We start with the following lemma.
 \begin{lem}\label{extac}
Let $\delta$ be an algebraic character on $\mathsf H$ (defined
over $\rk$). If the character $\abs{\delta}_\rk:H\rightarrow
\C^\times$ extends to a character on $G$, then $\delta^m$ extends to
an algebraic character on $\mathsf G$ for some positive integer
$m$.
 \end{lem}
\begin{proof}
We first assume that $\mathsf H$ is also connected.  Write
$\Psi_{\mathsf G}$ and $\Psi_{\mathsf H}$ for the groups of
algebraic characters of $\mathsf G$ and $\mathsf H$,
respectively. They are free abelian groups of finite rank. We
identify the group of positive characters on $G$ with
$\Psi_{\mathsf G}\otimes_\Z \R$ via the following isomorphism:
\[
  \Psi_{\mathsf G}\otimes_\Z \R\rightarrow \{\textrm{positive
  characters on $G$}\},\quad \delta\otimes a\mapsto
  \abs{\delta}_\rk^a.
\]
Likewise we identify the group of positive characters on $H$ with
$\Psi_{\mathsf H}\otimes_\Z \R$. Then we have a commutative
diagram
\[
  \begin{CD}
           \Psi_{\mathsf G}@>>>\Psi_{\mathsf G}\otimes_\Z \R\\
            @V\alpha_1 VV           @VV\alpha_2 V\\
          \Psi_{\mathsf H} @>>> \Psi_{\mathsf H}\otimes_\Z \R,
  \end{CD}
\]
where $\alpha_1$ denotes the map of restrictions of algebraic
characters, and $\alpha_2$  denotes the map of restrictions of
positive characters.

 Now let $\delta \in \Psi_{\mathsf H}$ and
assume that $\abs{\delta}_\rk$ extends to a character on $G$. Then
$\abs{\delta}_\rk$ extends to a positive character on $G$, that is,
it belongs to the image of $\alpha_2$. Then it is elementary
that
\[
  \abs{\delta}_\rk\in \alpha_2(\Psi_{\mathsf G}\otimes_\Z \BQ).
\]
Therefore $\abs{\delta^m}_\rk\in \alpha_2(\Psi_{\mathsf G})$ for
some positive integer $m$. Then $\delta^m\in
\alpha_1(\Psi_{\mathsf G})$ and the lemma is proved in the case
when $\mathsf H$ is connected.

Now we drop the assumption that $\mathsf H$ is connected. Let
$\delta \in \Psi_{\mathsf H}$ and assume that $\abs{\delta}_\rk$
extends to a character on $G$ as before. We have proved that there
exists an algebraic character $\delta'$ on $\mathsf G$ such that
\[
  \delta'|_{\mathsf H_0}=(\delta|_{\mathsf H_0})^m
\]
for some positive integer $m$, where $\mathsf H_0$ denotes the
identity connected component of $\mathsf H$. Then
\[
  {\delta'}^d|_{\mathsf H}=\delta^{md},
\]
where $d$ denotes the cardinality of the group $(\mathsf
H/\mathsf H_0)(\bar\rk)$, and $\bar \rk$ denotes an algebraic
closure of $\rk$. This finishes the proof of the Lemma.
\end{proof}

\begin{lem}\label{lalgm}
Assume that $X$ is $\chi$-admissible for some character $\chi$ on
$G$. Then there is an algebraic character $\delta$ of $\mathsf G$
and a positive integer $m$ such that there is a non-zero   algebraic
$m$-volume form $\omega$ on $\mathsf G/\mathsf H$ which is
$\delta$-invariant.
\end{lem}
\begin{proof}
Let $\Delta_{\mathsf G}$ denote the algebraic modular character of
$\mathsf G$, namely the determinant of the adjoint representation
of $\mathsf G$ on the Lie algebra $\Lie(\mathsf G)$. Likewise
let $\Delta_{\mathsf H}$ denote the algebraic modular character of
$\mathsf H$.
Put $\Delta_{\mathsf G/\mathsf
H}:=\frac{\Delta_{\mathsf G}|_{\mathsf H}}{\Delta_{\mathsf
H}}$, which is an algebraic character on $\mathsf H$.
Recall the
character $\delta_{H\backslash G}$ on $H$ from Section \ref{sfro}.
Note that
\be
\label{modular}
  \delta_{H\backslash G}=\abs{\Delta_{\mathsf G/\mathsf
H}}_\rk  .
\ee

Assume that $X$ is $\chi$-admissible for some character $\chi$ on
$G$. By \eqref{ado}, the character $\delta_{H\backslash G}$
extends to a character on $G$. Then lemma \ref{extac} implies that
the algebraic character $\Delta_{\mathsf G/\mathsf H}^{m}$
extends to an algebraic character $\delta$ on $\mathsf G$, for
some positive integer $m$. Then the lemma follows by the algebraic
version of Frobenius reciprocity.

\end{proof}

Now we come to the proof of Theorem \ref{defm}.
 Let $\eta\in
\Hom_{G,k }(\oS(X),\chi)$. We assume that $\eta$ is
non-zero. Then $X$ is $\chi$-admissible. By Lemma \ref{lalgm}, there
is an algebraic character $\delta$ on $\mathsf G$, a positive
integer $m$, and a   non-zero   algebraic $m$-volume form $\omega$ on
$\mathsf G/\mathsf H$ which is $\delta$-invariant. By Lemma
\ref{compfd}, there is a unique generalized function $f\in
\con^{-\infty}(X)$ such that $\eta=f\abs{\omega}_\rk^{1/m}$. Then
\[
 f\in \Hom_{G,k}(\oD_c^\infty(X), \chi\abs{\delta}_\rk^{\frac{-1}{m}}).
\]
Therefore Theorem \ref{defm} follows by Proposition \ref{H}.

\subsection{Definability of generalized invariant distributions}

First we have the following elementary lemma.
\begin{lem}\label{semigl1} Every finite index subgroup of
$\rk^\times$ is semialgebraic. Consequently, every finite order
character  on $\rk^\times$ is definable of order $\leq 0$.
\end{lem}
\begin{proof}
Every  subgroup of $\rk^\times$ of finite index $m\geq 1$ contains
$(\rk^\times)^m$. Since $(\rk^\times)^m$ is a semialgebraic subgroup
of $\rk^\times$ of finite index, the lemma follows.
\end{proof}

Recall the following semi-algebraic selection theorem of \cite{VdD}.
See also \cite[Appendix]{DV}.

\begin{lem}\label{ssel0}
Every surjective semialgebraic map of semialgebraic spaces has a
semialgebraic section.
\end{lem}

We continue with the notation of the last subsection, but drop the assumption that $\mathsf G$ is connected. Generalizing Lemma \ref{semigl1}, we have the following lemma.
\begin{lem}\label{semifc2}
Every finite order character $\chi_\mathrm f$ on $G$ is definable of
order $\leq 0$.
\end{lem}
\begin{proof}
Recall the multiplication map
\[
 \varphi: (\mathsf S(\rk)\times \mathsf T(\rk)\times \mathsf A(\rk))\ltimes N\to G=\mathsf G(\rk)
\]
from \eqref{decomg2}. Since the image of $\varphi$ is a
semialgebraic subgroup of $G$ of finite index. By Lemma \ref{ssel0},
it suffices to show that the finite order character $\chi_\mathrm
f':=\chi_\mathrm f\circ \varphi$ is definable of order $\leq 0$.
This is true because $\chi_\mathrm f'$ has trivial restriction to
$\mathsf S(\rk)$ and $N$, $(\chi_\mathrm f')|_{\mathsf T(\rk)}$
is a Bruhat-Schwartz function, and by Lemma \ref{semigl1},
$(\chi_\mathrm f')|_{\mathsf A(\rk)}$ is definable of order $\leq
0$.
\end{proof}

\begin{prpl}\label{defm3}
Let $\chi$ be a character on $G$ which is trivial on $N$. Then every
element of $\Hom_{G,k }(\oS(X),\chi)$ ($k\in \BN$) is a
definable measure on $X$ of order $\leq k$.
\end{prpl}
\begin{proof}
Without loss of generality, assume that $\mathsf G$ is connected.
Then the proposition follows by Theorem \ref{defm} and Lemma
\ref{semifc2}.
\end{proof}

\subsection{Locally finiteness of some algebraic measures}
\label{secalg}

For each algebraic variety $\mathsf X$ over $\rk$, write
$\mathsf X_{\mathrm{sm}}$ for the smooth part of $\mathsf X$,
which is an open subvariety of $\mathsf X$. Recall that a strong
resolution of singularities of $\mathsf X$ is a smooth algebraic
variety $\tilde {\mathsf X}$ (over $\rk$) together with a proper
birational morphism $\pi:\tilde {\mathsf X}\to \mathsf X$ such
that  $\pi: \pi^{-1}(\mathsf X_{\mathrm{sm}})\to \mathsf
X_{\mathrm{sm}}$ is an isomorphism. The famous theorem of Hironaka
says that $\mathsf X$ always has a strong resolution of
singularities.

Recall the following definition.

\begin{dfnl}\label{defrs}
We say that an algebraic  variety $\mathsf X$ over $\rk$ has
rational singularities if it is normal, and there exists a strong
resolution of singularities $\pi:\tilde {\mathsf X}\to \mathsf X$
of $\mathsf X$ such that the higher derived direct images vanish,
that is, $\CR^i\pi_*(\CO_{\tilde{\mathsf X}})=0$ for all $i>0$.
Here  $\CR^i\pi_*$ denotes the $i$-th derived functor of the
push-forward functor of sheaves via $\pi$.
\end{dfnl}

We will use the following property of algebraic varieties with
rational singularities.
\begin{lem}\label{algext}
Let $\mathsf X$ be an algebraic variety over $\rk$ with rational
singularities. Let $\mathsf U$ be a smooth open subvariety of
$\mathsf X$ whose complement has codimension $\geq 2$. Let
$\pi:\tilde {\mathsf X}\to \mathsf X$ be a strong resolution of
singularities. Then for each algebraic volume form $\omega$ on
$\mathsf U$,  there is an algebraic volume form $\tilde
\omega$ on $\tilde {\mathsf X}$ so that its restriction to
$\pi^{-1}(\mathsf U)$ is identical to $\omega$ via the isomorphism
$\pi: \pi^{-1}(\mathsf U)\to \mathsf U$.
\end{lem}
\begin{proof}
See \cite[P.50, Proposition]{KKMS} or \cite[Proposition 1.4]{ADK}.
\end{proof}

Given a measurable space $X$ with a measurable subset $Y$ of it, for
each measure $\mu$ on $Y$, we write $\mu|^X$ for the measure on $X$
which is obtained from $\mu$ by the extension by zero.

\begin{prpl}\label{lf0}(\cf \cite[Lemma 3.4.1]{AA})
Let  $\mathsf X$ be an algebraic variety over $\rk$ with rational
singularities. Let $\mathsf U$ be a smooth open subvariety of
$\mathsf X$ whose complement has codimension $\geq 2$. Then the
measure  ${(\abs{\omega}_\rk)}|^{\mathsf X(\rk)}$ is locally
finite for all algebraic volume form $\omega$ on $\mathsf U$.
\end{prpl}
\begin{proof}
Let $\pi:\tilde {\mathsf X}\to \mathsf X$ be a strong resolution
of singularities. Let  $\tilde \omega$ be as in Lemma \ref{algext}.
Write $\pi(\abs{\tilde \omega})$ for the push-forward of
$\abs{\tilde \omega}_\rk$ through the map \be\label{ppi}
  \pi: \tilde {\mathsf X}(\rk)\to \mathsf X(\rk).
\ee Then the measure $\pi(\abs{\tilde \omega})$ is locally finite
since \eqref{ppi} is a proper continuous map of topological spaces.
The proposition then follows by noting that $\pi(\abs{\tilde
\omega})-{(\abs{\omega}_\rk)}|^{\mathsf X(\rk)}$ is a non-negative
measure.
\end{proof}

\begin{dfnl}\label{defrs}
We say that an algebraic  variety $\mathsf X$ over $\rk$ has
Gorenstein rational singularities if it has rational singularities,
and the push forward $\mathcal{K}_{\mathsf X}$ of $\wedge^\mathrm{top} \Omega_{\mathsf
X_{\mathrm{sm}}}$ through the inclusion map $\mathsf
X_{\mathrm{sm}}\hookrightarrow \mathsf X$ is a locally free
$\CO_{\mathsf X}$-module.
\end{dfnl}
The sheaf $\mathcal{K}_{\mathsf X}$ is called the dualizing sheaf of $\mathsf X$.   If $\mathsf X$ is smooth, then $\mathcal{K}_{\mathsf X}\simeq \wedge^{\mathrm{top}}\Omega_{\mathsf{X}}$.  We have the following examples of Gorentain rational singularities:

\begin{enumerate}
\item If $\mathsf X$ has symplectic singularities, then $\mathsf X$ has rational Gorenstein singularities, see \cite[Proposition 1.3]{Bea}.
\item The normalization
of nilpotent varieties
in semisimple Lie algebras have  Gorenstein rational singularities,  see \cite{Hin}.
\end{enumerate}

Recall the following standard fact in algebraic geometry.
\begin{lem}
\label{Hartog}
Let $\mathsf X$ be a normal variety and suppose that $\CF$ is a locally free sheaf on
$\mathsf X$.
Let $\mathsf U$ be an open subvariety of $\mathsf X$ such that the complement $\mathsf X\setminus
\mathsf U$ is of codimension $\geq 2$. Then the restriction map $\Gamma(\mathsf X,\CF)\to
\Gamma(\mathsf U,\CF)$ is an isomorphism.
\end{lem}
\begin{proof}
See \cite[Proposition 1.11, Theorem 1.9]{Ha2}.
\end{proof}

We say that a quasi-coherent sheaf $\mathcal{F}$ on an algebraic variety $\mathsf X$ is torsion-free if for every $x\in \mathsf X$, the stalk $\mathcal{F}_x$ is torsion-free as a module of the local ring $\CO_{\mathsf X, x}$.
The following fact is standard. We omit its easy proof.
\begin{lem}
\label{Unique_Extension}
Let $\CF$ be a torison-free quasi-coherent sheaf on an algebraic variety $\mathsf X$. Let $\mathsf U$  be an open
subset of $\mathsf X$ whose complement has codimension $\geq 1$.  Then the restriction map $\CF(\mathsf X)\to \CF(\mathsf U)$ is injective.
\end{lem}
Similar to Lemma \ref{algext}, we have the following proposition.
\begin{prpl}\label{algext2}
Let $\mathsf X$ be an algebraic variety over $\rk$ with Gorenstein
rational singularities. Let $\mathsf U$ be a smooth open subvariety
of $\mathsf X$ whose complement has codimension $\geq 2$.  Let
$\pi:\tilde {\mathsf X}\to \mathsf X$ be a strong resolution of
singularities. Then for every positive integer $m$ and every algebraic
$m$-volume form $\omega$ on $\mathsf U$,  there is an algebraic
$m$-volume form $\tilde \omega$ on $\tilde {\mathsf X}$ so that
its restriction to $\pi^{-1}(\mathsf U)$ is identical to $\omega$
via the isomorphism $\pi: \pi^{-1}(\mathsf U)\simeq \mathsf U$.
\end{prpl}
\begin{proof}

Let $\omega$ be an algebraic $m$-volume form on $\mathsf U$, that is, $\omega\in \Gamma(\mathsf U, (\wedge^{\mathrm{top}}\Omega_{\mathsf{U}})^{\otimes
m} )$.
We  choose an affine open covering $\{\mathsf V_\alpha\}_{\alpha\in I}$  of $\mathsf X$.  Let $ \omega_\alpha$
be the restriction of $\omega$ to $\mathsf U  \cap \mathsf V_\alpha$ for each $\alpha$.
Since $\mathsf V_\alpha$ is affine, by \cite[Proposition 5.2(b), Chapter II]{Ha}  we have
$\Gamma(\mathsf V_\alpha,(\mathcal{K}_{\mathsf V_\alpha})^{\otimes m})\simeq \Gamma(\mathsf V_\alpha,\mathcal{K}_{\mathsf V_\alpha})^{\otimes
m}.$
For each $\alpha$ we have the following isomorphisms:
\begin{eqnarray*}
\Gamma(\mathsf U  \cap \mathsf V_\alpha,\wedge^\mathrm{top}
\Omega_{\mathsf U  \cap \mathsf V_\alpha})^{\otimes m} &\simeq& \Gamma({\mathsf V}_\alpha,\mathcal{K}_{{\mathsf V}_\alpha})^{\otimes
m} \\
&\simeq &
\Gamma(\mathsf V_\alpha,(\mathcal{K}_{{\mathsf V}_\alpha})^{\otimes m}) \\
  &\simeq&  \Gamma(\mathsf U  \cap \mathsf V_\alpha,(\wedge^\mathrm{top}
\Omega_{
\mathsf U  \cap \mathsf V_\alpha})^{\otimes
m}),
\end{eqnarray*}
where the the first and the last isomorphisms follow from Lemma \ref{Hartog}.

Therefore $\omega_\alpha$ can be expressed as a finite sum
$$\omega _\alpha =\sum_{i=1}^{n_\alpha} \omega_{\alpha,i,1}\otimes \omega_{\alpha,i,2}\otimes \cdots \otimes
 \omega_{\alpha,i,m} \qquad(n_\alpha\geq 0),$$
where each $\omega_{\alpha,i,k}$ is an algebraic volume  form
on $\mathsf U  \cap \mathsf V_\alpha$.

  By Lemma \ref{algext} and Lemma \ref{Unique_Extension}, the pull-back $\pi^*\omega_{\alpha,i, k}$
is uniquely extended to an algebraic volume form   $\tilde \omega_{\alpha,i,k}$ on $\pi^{-1}(\mathsf V_\alpha)$. Put
 $$
 \tilde \omega_\alpha:=\theta_\alpha\left(\sum_{i=1}^{n_\alpha}
\tilde \omega_{\alpha,i,1}\otimes         \tilde \omega_{\alpha,i,2}\otimes
\cdots \otimes  \tilde \omega_{\alpha,i,m}\right), $$
where $\theta_\alpha$ is the natural map
$$\theta_\alpha: \Gamma(\pi^{-1}(\mathsf V_\alpha),\wedge^\mathrm{top}
\Omega_{\pi^{-1}(\mathsf V_\alpha)})^{\otimes m}\to
\Gamma(\pi^{-1}(\mathsf V_\alpha), (\wedge^\mathrm{top}
\Omega_{\pi^{-1}(\mathsf V_\alpha)})^{\otimes m}).$$
It is clear that $\tilde \omega_\alpha$ is an extension of $\pi^*(\omega_{\alpha})$
from $\pi^{-1}(\mathsf U  \cap \mathsf V_\alpha)$ to $\pi^{-1}(\mathsf V_\alpha)$.

By Lemma \ref{Unique_Extension}, for all $\alpha,\beta\in I$, the two algebraic $m$-volume forms  $\tilde \omega_{\alpha}$ and $\tilde \omega_{\beta}$ coincide on
$\pi^{-1}(\mathsf V_\alpha\cap \mathsf V_\beta)$, since they coincide on $\pi^{-1}(\mathsf V_\alpha\cap \mathsf V_\beta\cap U)$. Hence $\{\tilde \omega_\alpha\}$ can be glued
to be an algebraic  $m$-volume  form $\tilde \omega$ on $\tilde {\mathsf X}$, and clearly the restriction of $\tilde \omega$ to $\pi^{-1}(\mathsf U)$ is identical to $\omega$
via the isomorphism $\pi: \pi^{-1}(\mathsf U)\simeq \mathsf U$.

\end{proof}

Similar to Proposition \ref{lf0}, we have the following proposition.
\begin{prpl}\label{convg_m_volume}
Let  $\mathsf X$ be an algebraic variety over $\rk$ with Gorenstein
rational singularities. Let $\mathsf U$ be a smooth open subvariety
of $\mathsf X$ whose complement has codimension $\geq 2$. Then the
measure ${(\abs{\omega}_\rk^{1/m})}|^{\mathsf X(\rk)}$ is locally
finite for all algebraic $m$-volume form $\omega$ on $\mathsf U$
($m\geq 1$).
\end{prpl}
\begin{proof}
The proof is similar to that of Proposition \ref{lf0}.
\end{proof}

\subsection{Locally finiteness of generalized invariant measures}
As before, let $\mathsf G$ be a linear algebraic group over $\rk$,
with unipotent radical $\mathsf N$. Put $G:=\mathsf G(\rk)$ and
$N:=\mathsf N(\rk)$. Let $\chi$ be a character on $G$ which is
trivial on $N$. In this subsection, we prove the following theorem.
\begin{thml}\label{convg0}
Let  $\mathsf X$ be an algebraic variety over $\rk$ of Gorenstein
rational singularity. Let $\mathsf U$ be a smooth open subvariety
of $\mathsf X$ whose complement has codimension $\geq 2$. Assume
that $\mathsf U$ is a homogeneous space of $\mathsf G$. Let
$\eta$ be a $\chi$-generalized invariant distribution on $\mathsf
U(\rk)$. Then $\eta$ is a measure, and $\eta|^{\mathsf X(\rk)}$ is
locally finite.
\end{thml}

\begin{lem}\label{lfinv}
Theorem \ref{convg0} holds when $\mathsf G$ is connected.
\end{lem}
\begin{proof}
We are in the setting of Theorem \ref{convg0} and assume that
$\mathsf G$ is connected. Theorem \ref{convg0} is trivial when
$\mathsf U(\rk)$ is empty. So assume that $\mathsf U(\rk)$ is
non-empty. Then we may (and do) assume that $\mathsf U=\mathsf
G/\mathsf H$ for some algebraic subgroup $\mathsf H$ of
$\mathsf G$. Since $(\mathsf G/\mathsf H)(\rk)$ is the disjoint
union of finitely many $G$-open orbits, we assume without loss of
generality that the distribution $\eta$ is supported on $G/H$. Write
\[
  \eta|_{G/H}= P(\mathrm{val}\circ \alpha_1, \mathrm{val}\circ \alpha_2, \cdots, \mathrm{val}\circ \alpha_r) \cdot \abs{\beta_1}_\rk^{s_1} \cdot \abs{\beta_2}_\rk^{s_2}
   \cdot \cdots \cdot
   \abs{\beta_t}_\rk^{s_t} \cdot \chi_\mathrm f \cdot (\abs{\omega}_\rk^{1/m})|_{G/H},
\]
as in Theorem \ref{defm}. By Lemma \ref{Hartog}, $\alpha_1,
\alpha_2, \cdots, \alpha_r$, $\beta_1, \beta_2,\cdots, \beta_t$
extend to elements of $\CO(\mathsf X)$. Therefore
\[
  P(\mathrm{val}\circ \alpha_1, \mathrm{val}\circ \alpha_2, \cdots, \mathrm{val}\circ \alpha_r) \cdot \abs{\beta_1}_\rk^{s_1} \cdot \abs{\beta_2}_\rk^{s_2}
   \cdot \cdots \cdot
   \abs{\beta_t}_\rk^{s_t}
\]
extends to a continuous function on $\mathsf X(\rk)$. By
Proposition \ref{convg_m_volume}, ${(\abs{\omega}_\rk^{1/m})}|^{\mathsf
X(\rk)}$ is locally finite. Therefore the lemma follows.
\end{proof}

Denote by $\mathsf G_0$ the identity connected component of
$\mathsf G$.
\begin{lem}\label{g0homo}
Let $\mathsf U$ be a homogeneous space of $\mathsf G$. Then
every connected component of $\mathsf U$ containing an $\rk$-point
is $\mathsf G_0$-stable and homogeneous.
\end{lem}
\begin{proof}
Let $x_0\in \mathsf U(\rk)$. Let $\mathsf H$ denote the
stabilizer of $x_0$ in $\mathsf G$. Then $\mathsf U=\mathsf
G/\mathsf H$. Note that $\mathsf G_0\cdot \mathsf H$ is an
open algebraic subgroup of $\mathsf G$, and $(\mathsf G_0\cdot
\mathsf H)/\mathsf H=\mathsf G_0/(\mathsf G_0\cap \mathsf
H)$ is a connected homogeneous space of $\mathsf G_0$. Write
\[
  \mathsf
G/\mathsf H= (\mathsf G_0\cdot \mathsf H)/\mathsf H
\,\sqcup\, (\mathsf G\setminus (\mathsf G_0\cdot \mathsf
H))/\mathsf H,
\]
and the lemma follows.
\end{proof}

Now we come to the proof of Theorem \ref{convg0} in general.
 Since $\mathsf X$ is normal, it is the disjoint
union of its irreducible components, and all the irreducible
components are open in $\mathsf X$. We only need to show that for
every irreducible component $\mathsf X'$ of $\mathsf X$,
\be\label{rup}
 \textrm{ $\eta':=\eta|_{\mathsf U'(\rk)}$ is a
measure, and $\eta'|^{\mathsf X'(\rk)}$ is locally finite,} \ee
where $\mathsf U':=\mathsf X'\cap \mathsf U$. This is
trivially true if $\mathsf U'(\rk)$ is empty. So assume that
$\mathsf U'(\rk)$ is non-empty. Then $\mathsf U'$ is an
irreducible component of $\mathsf U$, and hence it is also a
connected component of $\mathsf U$ since $\mathsf U$ is normal.
By Lemma \ref{g0homo}, $\mathsf U'$ is $\mathsf G_0$-stable and
homogeneous. Therefore \eqref{rup} holds by Lemma \ref{lfinv}.

\subsection{Proof of Theorem \ref{main2}}
Now we are in the setting of Theorem \ref{main2}. Let $\eta$ be a
$\chi$-generalized invariant distribution on $\mathsf U(\rk)$. By
Theorem \ref{convg0}, $\eta$ is a measure and the measure
$\eta|^{\mathsf X_f(\rk)}$ is locally finite. Proposition
\ref{defm3} implies that the measure $\eta|^{\mathsf X_f(\rk)}$ is
definable of order $\leq k$ for some $k\in \BN$. By Corollary
\ref{corextm}, $\eta|^{\mathsf X_f(\rk)}$ extends to a
$\chi$-generalized invariant distribution on $\mathsf X(\rk)$.
This finishes the proof of Theorem \ref{main2}.

\subsection{A variant of Theorem \ref{main2} }
We also have the following theorem.
\begin{thml}
\label{variant2}
Let $\mathsf G$ be a linear algebraic group over $\rk$. Let
$\mathsf X$ be  an algebraic variety over $\rk$  so that
$\mathsf G$ acts algebraically  on it with an open orbit
$\mathsf U\subset \mathsf X$. Assume that there is a non-zero semi-invariant algebraic volume form on $\mathsf U$, and  there is a  semi-invariant
regular function $f$ on  $\mathsf X$   with the
following properties:
\begin{itemize}
  \item $f$ does not vanish on $\mathsf U$, and $\mathsf X_f\setminus \mathsf U$ has codimension $\geq 2$ in  $\mathsf X_f$, where $\mathsf X_f$ denotes the complement  in $\mathsf X$ of the zero locus of $f$;
  \item the variety $\mathsf X_f$ has  rational singularities.
\end{itemize}
 Let $\chi$ be a character of $\mathsf G(\rk)$ which is trivial on $\mathsf N(\rk)$, where $\mathsf N$ denotes the unipotent radical of $\mathsf G$.
Then every generalized $\chi$-invariant distribution on $\mathsf
U(\rk)$ extends to a  generalized $\chi$-invariant distribution on
$\mathsf X(\rk)$.
\end{thml}

The proof of Theorem \ref{variant2} is the same as that of Theorem \ref{main2}, except that we should replace Theorem \ref{convg0} by the following theorem.
\begin{thml}
\label{r-measure}
Let  $\mathsf X$ be an algebraic variety over $\rk$ of
rational singularities. Let $\mathsf U$ be a smooth open subvariety
of $\mathsf X$ whose complement has codimension $\geq 2$. Assume
that $\mathsf U$ is a homogeneous space of $\mathsf G$ and there exists a non-zero semi-invariant algebraic volume form on $\mathsf U$. Let
$\eta$ be a $\chi$-generalized invariant distribution on $\mathsf
U(\rk)$, where $\chi$ is as in Theorem \ref{variant2}. Then $\eta$ is a measure, and $\eta|^{\mathsf X(\rk)}$ is
locally finite.
\end{thml}
The proof of Theorem \ref{r-measure} is also similar to that of Theorem \ref{convg0}, except that we replace Proposition \ref{convg_m_volume} by Proposition \ref{lf0}.

\section{Generalized semi-invariant distributions on matrix spaces}
\label{example}

We consider the following action of $G:={\rm GL}_m(\rk)\times {\rm GL}_n(\rk)$ ($m,n\geq 1$) on the  space ${\rm M}_{m,n}:={\rm M}_{m,n}(\rk)$ of $m\times n$-matrices with coefficients in $\rk$:
$$(g_1,g_2)\cdot x:=g_1 x g_2^{-1}, \qquad \text{ } g_1\in {\rm GL}_m(\rk),\, g_2\in {\rm GL}_n(\rk), \,x\in {\rm M}_{m,n}. $$
 For $r=0,1,\cdots, \min\{m,n\}$, let $\rm O_r$ denote the set of rank $r$ matrices in ${\rm M}_{m,n}$, which is a $G$-orbit.     Put $\bar{\rm O}_r:=\bigsqcup_{i=0}^r \rm O_i$.   Then ${\rm O}_r$ is open and dense in $\bar{\rm O}_r$.
Every character  of $G$ is given by
$$(g_1,g_2)\mapsto  \chi_1({\rm det}(g_1))\chi_2({\rm det}(g_2)), $$
for some characters $\chi_1,\chi_2$ of $\rk^\times$.  We denote this character of $G$ by the pair $(\chi_1,\chi_2)$.       Let ${\rm I}_r=(a_{ij})$ be the matrix in ${\rm O}_r$ such that
$$a_{ij}=\begin{cases}
1   \qquad \text{ if }  1\leq i=j\leq  r, \\
0  \qquad \text{ otherwise }.
\end{cases}  $$
The stabilizer group $G_r$ of ${\rm I}_r$ in $G$ consists of elements of the form
$$\left (\begin{bmatrix}  x   &   y   \\   0 &  w_1  \end{bmatrix},   \begin{bmatrix}  x   &   0   \\   z &  w_2  \end{bmatrix}     \right ),$$
where    $ x\in {\rm GL}_r(\rk), y\in {\rm M}_{r,m-r}(\rk), z\in {\rm M}_{n-r,r}(\rk), w_1\in {\rm GL}_{m-r}(\rk), w_2\in {\rm GL}_{n-r}(\rk)  $.
The algebraic modular character $\Delta_{G_r}$ of $G_r$ is given by
$$\left (\begin{bmatrix}  x   &   y   \\   0 &  w_1  \end{bmatrix},   \begin{bmatrix}  x   &   0   \\   z &  w_2  \end{bmatrix}     \right )  \mapsto  {\rm det}(x) ^{m-n}  {\rm det}(w_1)^{-r}  {\rm det} (w_2)^{r}. $$
By (\ref{ado}) and (\ref{modular}), for each character $\und \chi$ of $G$, the orbit ${\rm O}_r$ is $\und \chi$-admissible  if and only if
\[
\und \chi|_{G_r}=|\Delta_G|_\rk\cdot |\Delta_{G_r}|_\rk^{-1}=|\Delta_{G_r}|_\rk^{-1}.
\]
Since the stabilizer $G_r$ is connected when viewed as an algebraic group, the orbit ${\rm O}_r$ is $\und \chi$-admissible if and only if it is weakly $\und \chi$-admissible.

\begin{proposition}
\label{case1}
Fix a character  $\und{\chi}=(\chi_1,\chi_2)$  of $G$. Assume that $m\not=n$, then the following holds.

(a) If $\und{\chi}=(1,1)$, then $\oD({\rm M}_{m,n})^{\und{\chi}, \infty}=\oD({\rm M}_{m,n})^{\und{\chi}}=\C\cdot \delta_0$, where $\delta_0$ is the delta distribution supported at $0$.

 (b) If $\und{\chi}=(|\cdot|_\rk^n, |\cdot|_\rk^{-m})$,  then $\oD({\rm M}_{m,n})^{\und{\chi}, \infty}=D(M_{m,n})^{\und{\chi}}=\C\cdot \mu_{{\rm M}_{m,n}}$, where $\mu_{{\rm M}_{m,n}}$ is a Haar measure on ${\rm M}_{m,n}$.

(c) If $\und{\chi}\not= (1,1),\, (|\cdot|_\rk^n, |\cdot|_\rk^{-m})$, then $\oD({\rm M}_{m,n})^{\und{\chi}, \infty}=0$.

\end{proposition}
\begin{proof}
Note that there is no non-constant semi-invariant regular function on each orbit ${\rm O}_r$. By Theorem \ref{defm}, all generalized $\und \chi$-invariant distributions on ${\rm O}_r$ are $\und \chi$-invariant.

If $\und{\chi}= (1,1)$, then $\rm O_0$ is the only $\und \chi$-admissible orbit, and therefore
\[
\oD({\rm M}_{m,n})^{\und{\chi}, \infty}=\oD({\rm M}_{m,n})^{\und{\chi}}=\C\cdot \delta_0.
\]

If $\und{\chi}=(|\cdot|_\rk^n, |\cdot|_\rk^{-m})$,  then ${\rm O}_{\min\{m,n\}}$ is the only $\und \chi$-admissible orbit. Then Theorem \ref{main1} implies that
\begin{eqnarray*}
  &&   \C \cdot  \mu_{{\rm M}_{m,n}}\subset \oD({\rm M}_{m,n})^{\und{\chi}}\subset \oD({\rm M}_{m,n})^{\und{\chi}, \infty}\\
  &=&\oD({\rm O}_{\min\{m,n\}})^{\und{\chi},\infty}=\oD({{\rm O}}_{\min\{m,n\}})^{\und{\chi}}=\C\cdot  \mu_{{\rm M}_{m,n}}.
\end{eqnarray*}

If $\und{\chi}\not= (1,1), \,(|\cdot|_\rk^n, |\cdot|_\rk^{-m})$, then each orbit ${\rm O}_r$ is not $\und \chi$-admissible. Therefore $\oD({\rm M}_{m,n})^{\und{\chi}, \infty}=0$.

\end{proof}

We now assume that $m=n$,  and we  denote by ${\rm M}_n$ the space ${\rm M}_{n,n}$.
Consider the following  zeta integral
$$\oZ_{\chi}(\phi,s)=\int_{{\rm M}_n}  \phi(x) \chi({\rm det}(x))|{\rm det}(x)|_\rk^{s}\frac{\od\!x}{|{\rm det}(x)|_\rk^n},$$
where $\rm det$ is the determinant function on ${\rm M}_{n}$, $\od\!x$ is the Haar measure on ${\rm M}_{n}$ so that the space ${\rm M}_n(\rr)$ of integral matrices in $\rm M_n$ has volume $1$, $\chi$ is a   character of $\rk^\times$ and $\phi\in \oS({\rm M}_n)$.  By Theorem \ref{igusa},  it is a rational function of $1-q_\rk^{-s}$.  Let  $\oZ_{\chi, i}$ be the $i$-th coefficient of the Laurent expansion of $\oZ_{\chi}$ (as a rational function of $1-q_\rk^{-s}$).  By Proposition \ref{Zeta_invariant},  $\oZ_{\chi,i}$ is a generalized $(\chi, \chi^{-1})$-invariant distribution. It is easy to check that
\be\label{actionzg}
  (1-g)\cdot \oZ_{\chi,i}=\oZ_{\chi,i-1},
\ee
for all $g=(g_1, g_2)\in G$ such that $\det (g_1^{-1} g_2)$ is a uniformizer of $\rr$.   Here the action of $G$ on $\oD({\rm M}_{n})^{(\chi,\chi^{-1}), \infty}\subset \Hom_\C(\oS(\rm M_n), (\chi, \chi^{-1}))$ is as in  the equation \eqref{actint} of Section \ref{secgh}.
\begin{proposition}
\label{case2}
(a) If  $\chi= |\cdot|_\rk^r$ for some $r=0, 1, \cdots, n-1$, then $\oZ_{\chi, i}=0$ for all $i<-1$, and $\{\oZ_{\chi, i}\}_{i\geq -1}$ is a basis of $\oD({\rm M}_{n})^{(\chi,\chi^{-1}), \infty}$.

(b) If  $\chi\neq  |\cdot|_\rk^r$ for all $r=0, 1, \cdots, n-1$, then $\oZ_{\chi, i}=0$ for all $i<0$, and $\{\oZ_{\chi, i}\}_{i\geq 0}$ is a basis of $\oD({\rm M}_{n})^{(\chi,\chi^{-1}), \infty}$.

(c) For every character $(\chi_1, \chi_2)$ of $G$, the space $\oD({\rm M}_{n})^{(\chi_1,\chi_2), \infty}=0$ if $\chi_1\chi_2\not=1$.

\end{proposition}
\begin{proof}

Note that for each $i<0$, $\oZ_{\chi,i}$ is supported in $\bar{\rm O}_{n-1}$, in other words,
\[
  \oZ_{\chi,i}\in \oD(\bar{\rm O}_{n-1})^{(\chi,\chi^{-1}), \infty}.
\]
It is  easy to see $\{\oZ_{\chi,i}|_{{\rm O}_n}\}_{i\geq 0}$ is a basis of  $\oD({\rm O}_n)^{(\chi,\chi^{-1}), \infty}$. In particular, we have an exact sequence
\be\label{exactod}
  0\rightarrow \oD(\bar {\rm O}_{n-1})^{(\chi,\chi^{-1}), \infty}\rightarrow \oD({\rm M}_n)^{(\chi,\chi^{-1}), \infty}\rightarrow \oD({\rm O}_n)^{(\chi,\chi^{-1}), \infty}\rightarrow 0.
\ee
If  $\chi\neq  |\cdot|_\rk^r$ for all $r=0, 1, \cdots, n-1$, then any orbit in $\bar{\rm O}_{n-1}$ is not $\chi$-admissible, i.e. not weakly $\chi$-admissible. By
 Bernstein-Zelevinsky localization principle,
\[
\oD(\bar{\rm O}_{n-1})^{(\chi,\chi^{-1}), \infty}=0.
\]
Therefore part (b) of the proposition follows.

Now assume that $\chi= |\cdot|_\rk^r$ ($r=0, 1, \cdots, n-1$). Then
 \be\label{odbar}
 \oD(\bar{\rm O}_{n-1})^{(\chi,\chi^{-1}), \infty}=\oD(\bar{\rm O}_r)^{(\chi,\chi^{-1}), \infty}=\oD({\rm O}_r)^{(\chi,\chi^{-1}), \infty}=\oD({\rm O}_r)^{(\chi,\chi^{-1})}.
 \ee
Here the first equality follows from the localization principle of Bernstein-Zelevinsky, the second one is implied by Theorem \ref{main1}, and the last one follows as in the proof of Proposition \ref{case1}. In particular, $\oZ_{\chi,i}$ is $(\chi,\chi^{-1})$-invariant for all $i<0$. Then \eqref{actionzg}  implies that that $\oZ_{\chi,i}=0$ for all $i< -1$. On the other hand, the computation (\cite[Chapter 10.1]{Ig})
$$\int_{{\rm M}_n(\rr)}  |{\rm det}(x)|^s \od\!x=\prod_{i=1}^n  \frac{1-q_F^{-i}}{1-q_F^{-i-s}} $$
implies that $\oZ_{\chi,-1}\neq 0$. Therefore $\oZ_{\chi,-1}$ is a generator of the one-dimensional space \eqref{odbar}. Now part (a) of the proposition follows by the exact sequence \eqref{exactod}.

Part (c) of the proposition is an easy consequence of Bernstein-Zelevinsky localization principle, since under which condition every orbit in ${\rm M}_n$ is not  $\chi$-admissible.

\end{proof}

In view of \eqref{actionzg}, Proposition \ref{case2} implies that
\[
  \dim \oD({\rm M}_{n})^{(\chi,\chi^{-1})}=1
\]
for all   character $\chi$ of $\rk^\times$. This generalizes the equality \eqref{dimtate} of Tate's thesis, and is a (well-known) particular case of local theta correspondence.

\end{document}